\documentclass[a4]{amsart}

\newif\ifMOR
\MORfalse

\ifMOR
    \input preambleMOR
\else
    \usepackage{amsaddr}
\usepackage[backend=bibtex]{biblatex}
\usepackage{a4}
\usepackage[utf8]{inputenc}
\usepackage[T1]{fontenc}
\usepackage{url}
\usepackage{amsmath}
\usepackage{amsfonts}
\usepackage{amssymb}
\usepackage{todonotes}
\setuptodonotes{fancyline}
\usepackage{xcolor}
\usepackage{xifthen}
\usepackage{tikz}
\usetikzlibrary {quotes} 
\usepackage{pgffor}
\usepackage{tikzit}

\definecolor{nontermcol}{RGB}{0,174,219}
\definecolor{termcol}{RGB}{0,177,89}

\definecolor{col0}{RGB}{31,119,180}
\definecolor{col1}{RGB}{255,127,14}
\definecolor{col2}{RGB}{44,160,44}
\definecolor{col3}{RGB}{214,39,40}
\definecolor{col4}{RGB}{148,103,189}
\definecolor{col5}{RGB}{140,86,75}
\definecolor{col6}{RGB}{227,119,194}
\definecolor{col7}{RGB}{127,127,127}
\definecolor{col8}{RGB}{188,189,34}
\definecolor{col9}{RGB}{23,190,207}

\definecolor{laminarScol}{RGB}{31,119,180}
\definecolor{laminarLcol}{RGB}{188,189,34}
\definecolor{laminarLLcol}{RGB}{227,119,194}

\tikzset{nonterm/.style={circle, draw=black, line width = .8mm, minimum size=.7cm}}
\tikzset{term/.style={circle, fill=gray!70, draw=black, line width = .8mm, minimum size=.7cm}}
\tikzset{edge1/.style={color = gray,line width = 4}}
\tikzset{edge2/.style={color = gray,line width = 8}}
\tikzset{edge3/.style={color = gray,line width = 12}}

\tikzset{laminarS/.style={color = laminarScol,line width = 2,rounded corners = 8}}
\tikzset{laminarL/.style={color = laminarLcol,line width = 2,rounded corners = 8}}
\tikzset{laminarLL/.style={color = laminarLLcol,line width = 2,rounded corners = 8}}
\newcommand{\laminarUnit}{.5}
\newcommand{\laminarOff}{.5}

\tikzset{pp_node/.style={circle, draw=black, line width = .8mm, minimum size=.7cm}}
\tikzset{pp_edge/.style={draw=black, line width = .8mm}}
\tikzset{case_term/.style={circle, fill=gray!70, draw=black, line width = .8mm, minimum size=.7cm}}
\tikzset{case_nonterm/.style={circle, draw=black, line width = .8mm, minimum size=.7cm}}
\tikzset{case_edge/.style={draw=black, line width = .8mm}}

\usepackage{soul}

\newif\ifINDICATECHANGES
\INDICATECHANGESfalse

\ifINDICATECHANGES
    \newcommand{\changescolor}{olive}
    \usepackage[defaultcolor=\changescolor]{changes}

    \makeatletter 
    \@mparswitchfalse%
    \makeatother
    \normalmarginpar 
\else
    \usepackage[final]{changes}
\fi

\setcommentmarkup{\todo[color={\changescolor!20},size=\scriptsize]{#1}}
\newcommand{\revrem}[2]{Rev{#1}:Rem{#2}}

\newcommand{\stkout}[1]{\ifmmode\text{\sout{\ensuremath{#1}}}\else\sout{#1}\fi}
\setdeletedmarkup{\stkout{#1}}

\definecolor{col0}{RGB}{31,119,180}
\definecolor{col1}{RGB}{255,127,14}
\definecolor{col2}{RGB}{44,160,44}
\definecolor{col3}{RGB}{214,39,40}
\definecolor{col4}{RGB}{148,103,189}
\definecolor{col5}{RGB}{140,86,75}
\definecolor{col6}{RGB}{227,119,194}
\definecolor{col7}{RGB}{127,127,127}
\definecolor{col8}{RGB}{188,189,34}
\definecolor{col9}{RGB}{23,190,207}

\newcommand{\zerovec}{\mathbb{O}}

\newcommand{\R}{\mathbb R}

\newcommand{\Z}{\mathbb Z}

\DeclareMathOperator{\cut}{CUT}
\newcommand{\cutd}{\cut_+}

\newcommand{\incvec}[1]{\chi({#1})}
\DeclareMathOperator{\bigO}{O}

\newcommand{\setdef}[2]{ \{ {#1} \,:\, {#2} \} }

\DeclareMathOperator{\YOp}{Y}
\newcommand{\YDred}{\(\YOp\!\nabla\)-reduction}

\newcommand{\scale}{1}

\makeatletter
 \if@todonotes@disabled
 
 \else
 
 \fi
 \makeatother

\graphicspath{{figures/}}

\title{Steiner Cut Dominants}
\author{Michele Conforti}
\address[Michele Conforti]{Dipartimento Matematica, Universit\`{a} degli Studi di Padova, Via Trieste 63, 35121 Padova, Italy}
\email{conforti@math.unipd.it}

\author{Volker Kaibel}
\address[Volker Kaibel]{Fakult\"at f\"ur Mathematik, OVGU Magdeburg, Universit\"atsplatz 2, 39106 Magdeburg, Germany}
\email{kaibel@ovgu.de}

    \newtheorem{theorem}{Theorem}
    \newtheorem{lemma}{Lemma}
    \newtheorem{proposition}{Proposition}
    \newtheorem{corollary}{Corollary}
    \newtheorem{remark}{Remark}

    \newtheorem{definition}{Definition}
    
\fi

\date{\today}

\begin{document}

\ifMOR
    \input titleInfoMOR
    \ABSTRACT{%
        	For a subset $T$  of nodes of an undirected graph $G$, a  \emph{$T$-Steiner cut} is a cut $\delta(S)$ with $T \cap S \ne \varnothing$ and $T \setminus S \ne \varnothing$. The \emph{$T$-Steiner cut dominant} of $G$ is the dominant $\cutd(G,T)$ of the convex hull of the incidence vectors of the $T$-Steiner cuts of $G$. For $T=\{s,t\}$, this is the  well-understood $s$-$t$-cut dominant.  Choosing $T$ as the set of all nodes of $G$, we obtain the \emph{cut dominant}, for which an outer description in the space of the original variables is still not known. We prove that, for each integer $\tau$, there is a finite set  of inequalities such that for every pair $(G,T)$ with $|T|\le\tau$ the non-trivial 
 \replaced{facet defining }{facet-defining} inequalities of $\cutd(G,T)$ are the inequalities that can be obtained via iterated applications of two simple operations, starting from that set. In particular,  the absolute values of the coefficients and of the right-hand-sides in a description of $\cutd(G,T)$  by integral inequalities can be bounded from above by a function of $|T|$. For all $|T| \le 5$ we  provide   descriptions of $\cutd(G,T)$ by facet defining inequalities, extending the known descriptions of $s$-$t$-cut dominants. 

    }%
    \maketitle
\else
    \maketitle
    \begin{abstract}
        	For a subset $T$  of nodes of an undirected graph $G$, a  \emph{$T$-Steiner cut} is a cut $\delta(S)$ with $T \cap S \ne \varnothing$ and $T \setminus S \ne \varnothing$. The \emph{$T$-Steiner cut dominant} of $G$ is the dominant $\cutd(G,T)$ of the convex hull of the incidence vectors of the $T$-Steiner cuts of $G$. For $T=\{s,t\}$, this is the  well-understood $s$-$t$-cut dominant.  Choosing $T$ as the set of all nodes of $G$, we obtain the \emph{cut dominant}, for which an outer description in the space of the original variables is still not known. We prove that, for each integer $\tau$, there is a finite set  of inequalities such that for every pair $(G,T)$ with $|T|\le\tau$ the non-trivial 
 \replaced{facet defining }{facet-defining} inequalities of $\cutd(G,T)$ are the inequalities that can be obtained via iterated applications of two simple operations, starting from that set. In particular,  the absolute values of the coefficients and of the right-hand-sides in a description of $\cutd(G,T)$  by integral inequalities can be bounded from above by a function of $|T|$. For all $|T| \le 5$ we  provide   descriptions of $\cutd(G,T)$ by facet defining inequalities, extending the known descriptions of $s$-$t$-cut dominants. 
	
    \end{abstract}
\fi

 \section{Introduction}

Let $(G,T)$ be a \emph{Steiner graph}, i.e.,  $G$ is a connected graph with node set $V(G)$ and edge set $E(G)$, and $T \subseteq V(G)$ with $|T|\ge 2$ is a subset of at least two \emph{terminals}. 
For $S \subseteq V(G)$ we denote by $\delta_G(S)$ or $\delta(S)$ the \emph{cut} in $G$ defined by $S$, i.e., the set of all edges with one endnode in $S$ and the other one  in $V\setminus S$, where we use $\delta(v):=\delta(\{v\})$ to denote the \emph{star} of a node $v \in V(G)$.
A cut $\delta(S)$ \deleted{in $T$ } is called a \emph{$T$-Steiner cut} if both $T\cap S\ne \varnothing$ and $T \setminus S \ne\varnothing$ are non-empty.
 The $T$-Steiner cut polytope $\cut(G,T)$ is the convex hull of the incidence vectors $\incvec{\delta(S)}\in\{0,1\}^{E(G)}$ of the $T$-Steiner cuts in $G$ (with $\incvec{\delta(S)}_e=1$ if and only if $e \in \delta(S)$ holds).
 The \emph{$T$-Steiner cut dominant} is defined to be the \replaced{Minkowski-sum }{polyhedron}
 \[
 	\cutd(G,T) =\cut(G,T)  + \R_{\ge 0}^{E(G)}
 \]
\added{of $\cut(G,T)$ and the nonnegative orthant, }
 i.e. the polyhedron formed by all points $y$ that dominate some point $x \in \cut(G,T)$ in the sense of $y \ge x$. The dominant  of a polyhedron~$P$ is the polyhedron that is essential for minimizing linear functions with nonnegative  coefficients over $P$.
 \added[comment={\revrem{2}{2}}]{It is the largest polyhedron with the property that for all nonnegative objective function vectors $c$ the infimum of $cx$ over it equals the corresponding infimum over $P$.}
 
 In case of  $T=\{s,t\}$ for some  $s,t \in V(G)$, the polyhedron  $\cutd(G,\{s,t\})$ is the well understood  \emph{$s$-$t$-cut dominant of $G$}, for which  
 \begin{equation}\label{eq:st_cut_dom}
 	\cutd(G,\{s,t\}) = 
 		\big\{x \in \R_{\ge\zerovec{}}^{E(G)}\,:\,
 				x(P) \ge 1 
 				\text{ for all $s$-$t$-paths }P \subseteq E(G)\big\}
 \end{equation}
 can be deduced  via the max-flow min-cut theorem and flow decomposition (for a closer investigation of $s$-$t$-cut dominants see Skutella and Weber~\cite{SkutellaWeber2010}). For more than two terminals, however,  inequality descriptions of the Steiner cut dominants have not been known so far.  In particular, for $T=V(G)$ the polyhedron 
 \[
    \cutd(G) = \cutd(G,V(G))
\]
 is  the much less understood \emph{cut dominant of $G$} (see, e.g., Conforti, Fiorini, and Pash\-ko\-vich~\cite{ConfortiFioriniPashkovich2016}), whose facets are still unknown. Despite the fact that this maybe makes up the most prominent example of a  polynomial time solvable combinatorial optimization problem (see the remarks below) for whose associated polyhedron no inequality description is known, understanding the facets of cut dominants seems desirable also because of the following relation.
 By blocking duality \added[comment={\revrem{2}{2}}]{(see, e.g., \cite[Sect. 9.2]{Schrijver86})} the coefficient vectors of the non-trivial 
 \replaced{facet defining }{facet-defining} inequalities of $\cutd(G)$ are the vertices of the \emph{graphical subtour relaxation polyhedron} associated with $G$ (see, e.g., \cite{ConfortiFioriniPashkovich2016}). Among those inequalities, the ones with the property that---when scaled such that the right-hand-side equals two---\replaced{for each node the components associated with the incident edges sum up to two }{the components sum up to two over each star} are the vertices of the \emph{subtour relaxation polytope} of $G$, which (usually considered for the  complete graph $K_n$ on $n$ nodes) makes up the most important relaxation of the \emph{traveling salesman polytope}. Interest in the vertices of the latter relaxation (thus, in the facets of $\cutd(K_n)$) has been raised, e.g., by the \emph{4/3-conjecture} on the quality of the \emph{subtour relaxation bound} for the metric traveling-salesman problem (see, e.g., Goemans~\cite{Goemans95}). 
  
 The \emph{minimum Steiner cut problem}, i.e.\added[comment={\revrem{2}{4}}]{,} the task to find for nonnegative edge weights a $T$-Steiner cut of minimal weight,  can obviously be solved in polynomial time by finding (in polynomial time) minimum $s$-$t$ cuts, e.g.,  by computing maximum $s$-$t$-flows, for $t=t_1,\dots,t_{|T|-1}$.  
 That optimization problem arises, for instance, as the separation problem associated with the inequalities in standard integer programming formulations for the \emph{Steiner tree problem} (see, e.g., Letchford, Saeideh, and Theis~\cite{LetchfordSaeidehTheis2013} or Fleischmann~\cite{Fleischmann1985}).  In fact, there has been great progress recently both w.r.t.  reducing the number of max-flow computations in the approach mentioned above (by Li and Panigrahi~\cite{LiPanigrahi2020})  as well as in solving the max-flow problem (by Chen, Kyng, Liu, Peng, Gutenberg, and Sachdeva~\cite{ChenKyngLiuPengGutenbergSachdeva2022}), resulting in a randomized algorithm for computing minimum Steiner cuts in graphs $G$ that runs in time $|E(G)|^{1+\text{o}(1)}$ with high probability (where the weights are assumed to be integral numbers whose sizes are bounded by a polynomial in $|E(G)|$, see also the remarks in the updated version of~\cite{LiPanigrahi2021} and the recent  improvement~\cite{DingL23}). For the minimum Steiner cut problem in \emph{planar}  graphs, Jue and Klein~\cite{JueKlein2019} designed a deterministic algorithm whose running time can be bounded by  $\bigO(|V(G)|\cdot\log|V(G)|\cdot\log|T|)$.

%  In contrast to the case of two terminals $T=\{s,t\}$, where 
%  \begin{equation}\label{eq:st_cut_dom}
%  	\cutd(G,\{s,t\}) = 
%  		\big\{x \in \R_{\ge\zerovec{}}^{E(G)}\,:\,
%  				x(P) \ge 1 
%  				\text{ for all $s$-$t$-paths }P \subseteq E(G)\big\}
%  \end{equation}
%  can be deduced easily via the max-flow-min-cut theorem and flow decomposition (see, e.g., \cite{SkutellaWeber2010}), for more than two terminals inequality descriptions of the Steiner cut dominants have not been known so far. This is somewhat in contrast to the fact that the \emph{minimum Steiner cut problem}, i.e. minimizing a nonnegative function  over $\cutd(G,T)$ whith $T=\{s, t_1,\dots,t_{|T|-1}\}$, can obviously be solved in polynomial time by computing (in polynomial time)  minimum $s$-$t$-cuts  for $t=t_1,\dots,t_{|T|-1}$ and  choosing the best one among them.  
%  The minimum Steiner cut problem arises, for instance, as the separation problem associated with the inequalities in standard integer programming formulations for the \emph{Steiner tree problem} (see, e.g., \cite{LetchfordSaeidehTheis2013,Fleischmann1985}).  In fact, recent work has 
 
 Using the fact that the set of $T$-Steiner cuts is the union of the sets of $s$-$t_i$-cuts for $|T|-1$ pairs $(s,t_i)$ and 
 following Balas' disjunctive programming paradigm~\cite{Balas2018} one can easily come up with an \emph{extended formulation} for $\cutd(G,T)$ based on~\eqref{eq:st_cut_dom}. When using flow-based extended formulations for the $s$-$t$-cut dominants instead of~\eqref{eq:st_cut_dom} one even obtains an extended formulation for $\cutd(G,T)$ with both $\bigO (|T| \cdot |E(G)|)$ variables and constraints. In fact, 
 striving for further improving the size of the representation (at least for dense graphs), Carr, Konjevod, Little, Natarajan, and Parekh~\cite{CarrKonjevodLittleNatarajanParekh2009}  introduced a polyhedron in $\bigO(|V|^2)$-dimensional space described by $\bigO(|T| \cdot |V|^2)$ many inequalities for which the dominant of its projection to the ambient space of $\cutd(G,T)$ equals $\cutd(G,T)$. Consequently, finding a minimum $T$-Steiner cut w.r.t. nonnegative weights can be done by solving a linear program over that polyhedron. (Actually, they provide a construction for $T=V(G)$ only, but that one can be readily generalized to arbitrary terminal sets.) 
 
 In this paper, however, we are not concerned with deriving algorithms for computing \replaced[comment={\revrem{2}{7}}]{minimum }{minmum} weight Steiner cuts or in designing extended formulations, but rather in the facets of Steiner cut dominants, i.e., we search for descriptions of Steiner cut dominants by means of inequalities for their ambient spaces $\R^{E(G)}$. 
  The difficulty of deriving such inequality descriptions may  become apparent from the fact that even for the cut dominants of planar graphs such descriptions  are not yet known, in contrast to the situation for the cut \emph{polytopes} of planar graphs which are described by the \replaced[comment={\revrem{2}{8}}]{\emph{cycle inequalities} }{\emph{cycle inqualities}} (see Barahona and Mahjoub~\cite{BarahonaMahjoub1986}).  
 
  One approach that has been taken in order to understand better the inequalities needed in descriptions of cut dominants is to classify them according to their right-hand-sides when scaled to be in \emph{minimum integer form}, i.e., such that  their non-zero coefficients are relatively coprime positive integers. Conforti, Fiorini, and Pashkovich~\cite{ConfortiFioriniPashkovich2016} derived a forbidden-minor characterization of the graphs $G$ for which $\cutd(G)$ has a description by inequalities with right-hand-sides at most two. 
   In fact, we will use that characterization later (see Theorem~\ref{thm:CFP}).

Besides appearing to be of independent interest, the concept of Steiner cut dominants offers another  approach to classify facet defining inequalities for cut dominants. In view of the fact that every inequality that is valid for $\cutd(G)$ is obviously valid for $\cutd(G,T)$ for every choice $T \subseteq V(G)$ of terminals (in general, for $T_1 \subseteq T_2$ we obviously have $\cutd(G,T_1) \subseteq \cutd(G,T_2)$), we define the \emph{Steiner degree} of an inequality defining a facet of $\cutd(G)$ to be the 
\replaced[comment={\revrem{2}{10}}]{minimum }{minimal} 
$\tau$ for which there is some $T \subseteq V(G)$ with $|T|=\tau$ such that the inequality defines a facet of $\cutd(G,T)$. Of course, this notion of \emph{Steiner degree} can be transfered readily to the vertices of subtour relaxation polyhedra and polytopes via the above mentioned blocking duality. 

If an inequality defining a facet of $\cutd(G,T)$ is  valid for $\cutd(G)$, then it clearly is a facet defining inequality for $\cutd(G)$ of Steiner degree at most $|T|$ (as both polyhedra are full-dimensional). For instance, the inequalities in~\eqref{eq:st_cut_dom} that arise from \emph{Hamiltonian} $s$-$t$-paths are facet defining inequalities for $\cutd(G)$ of Steiner degree two. A consequence (see Theorem~\ref{thm:bounded_rhs}) of our first main result is that the right-hand-side (and the coefficients) of any  facet defining  inequality in minimum integer form for a cut dominant is bounded from above by a function depending only on its Steiner degree. Note that the reverse of that statement does not hold, as, e.g.,  for every spanning tree $Q\subseteq E(G)$ of $G$  the inequality $x(Q) \ge 1$ defines a facet of $\cutd(G)$ whose  Steiner degree equals the number of leaves of $Q$ (which follows from Part~(2) of Proposition~\ref{prop:facets_first_properties} and the remarks at the beginning of Section~\ref{sec:trees_and_cacti}). A consequence of the second main result of our paper will be a classification of the 
facet defining inequalities for cut dominants of Steiner degree at most five (see Corollary~\ref{cor:Steiner_degree_at_most_five}).

Before we develop our results precisely in the following sections, we provide  informal  descriptions of them. 
\begin{description}
	\item[Main Result I.] 
	We discuss two operations (\emph{subdividing} and \emph{gluing}) that produce facet defining inequalities from facet defining inequalities and show that for each $\tau$ there is a finite set of inequalities from which, for every Steiner cut dominant with  $\tau$ terminals, a description by a system of facet defining  inequalities can be obtained by repeated applications of those two operations (see Remark~\ref{rem:reduction_to_irreducibles} and Theorem~\ref{thm:bounded_number_terminals_irreducible}).   
	\item[Main Result II.]
	We introduce two classes of inequalities (\emph{Steiner tree inequalities} and \emph{Steiner cactus inequalities}) for which we prove that each Steiner cut dominant with at most five terminals is described by the inequalities from those two classes (see Remark~\ref{rem:reduction_to_facet_graphs} and  Theorem~\ref{thm:facets_three_four_five}).  
\end{description}

The paper is structured in the following way. We start  in Section~\ref{sec:basics} by defining some notions and stating some basic results on Steiner cut dominants, most of which are well-known in the more general context of \emph{dominant} polyhedra  (i.e., polyhedra whose recession cones are non-negative orthants). In Section~\ref{sec:laminar}
 we establish a crucial structural result on every (non-trivial) facet of a Steiner cut dominant: It contains a subset of vertices whose associated Steiner cuts are induced by a laminar family of node sets with as many  members as an inequality describing the facet has  nonzero coefficients. This result generalizes a corresponding one for cut dominants\added[comment={\revrem{2}{9}}]{ (following from \cite[Thm. 4.9]{CornuejolsFN85} via blocking duality)}, but is somewhat more delicate to prove. The most important consequence of this laminarity result in our context is that the number of nonzero coefficients in inequalities defining facets of Steiner cut dominants is bounded from above by the number of nodes of the graph plus the number of terminals minus three. In Section~\ref{sec:split_glue} we then introduce the two operations \emph{subdividing} and \emph{gluing} referred to above, before we state and prove the two main results in Section~\ref{sec:main:I} and~\ref{sec:main:II}. We conclude with some open questions in Sect~\ref{sec:conclusion}.  

\added{%
\noindent\textbf{Acknowledgements.}
We would like to thank two anonymous reviewers for their  suggestions that helped to further improve the presentation of the material.
}
 
\section{Facet inducing Steiner graphs}
\label{sec:basics}

%By a 
%\todohl{\emph{Steiner graph}}{Seems convenient to have this notion available.} 
% we refer to a pair $(G,T)$ of a graph $G$ and a subset $T \subseteq V(G)$ with $|T|\ge 2$.

As the recession cone of $\cutd(G,T)$ is the nonnegative orthant, if $cx\ge \gamma$ is an inequality which is 
valid
for $\cutd(G,T)$, then $c$ is nonnegative.
We refer to  $c_e$ 
or $c(e)$
as the \emph{$c$-weight} of  edge $e$, and denote by  $\gamma_c(G,T)\ge \gamma$ the minimum $c$-weight of any $T$-Steiner cut  in $(G,T)$. We say that the vector $c$ is in \emph{minimum integer form} if its components are nonnegative integers whose greatest common divisor equals one. The \emph{minimum integer form} of a nonnegative rational non-zero vector is its unique scalar multiple that is in minimum integer form. 
We can (and always will) assume that  $cx\ge \gamma$, i.e. the coefficient vector $c$,  is in minimum integer form. 
We furthermore will always assume that 
$\gamma = \gamma_c(G,T) \ge 1$
holds (thus excluding nonnegativity constraints from our considerations). The $T$-Steiner cuts with $c$-weight equal to $\gamma_c(G,T)$ are called the \emph{roots} of $c$ and of the inequality. We sometimes refer to $\gamma_c(G,T)$ as the 
\emph{right-hand side} of $c$.

The graph $G_c=(V_c,E_c)$ is  the  subgraph of $G$  where $E_c$ is the set of all edges with non-zero (thus positive) $c$-weights, and $V_c$ is the set of all endpoints of edges in $E_c$. 
 We clearly have 
$\gamma_c(G,T) = \gamma_c(G_c,T)$.
The following statements  follows from the fact that $\cutd(G,T)$ is a dominant polyhedron in the nonnegative orthant that is full-dimensional.  Whenever we make statements about Steiner cuts in a linear algebra context those statements refer to the corresponding incidence vectors. 

\begin{remark} 
\label{rem:Gc}
For each  inequality $cx\ge \gamma=\gamma_c(G,T)$ that is valid  for $\cutd(G,T)$ the following hold:
\begin{enumerate}
    \item If $\gamma=0$, then the inequality defines a facet of $\cutd(G,T)$ 
    if and only if it is a nonnegativity constraint.
    \item If $\gamma>0$,  then the inequality defines a facet of $\cutd(G,T)$ if and only if $(G_c,T)$ has $|E_c|$ many $c$-minimum $T$-Steiner cuts that are 
    linearly independent (which is equivalent to them being affinely independent, since the affine hull of the face defined by $cx\ge\gamma$ does not contain the origin due to $\gamma\ne 0$).
\end{enumerate}
\end{remark}

If $cx\ge\gamma$  in minimum integer form defines a facet of $\cutd(G,T)$ with $G_c=G$ and  $\gamma=\gamma(G,T) > 0$, then  $(G,T)$ is called a \emph{facet inducing Steiner graph}, and we refer to $c$  as \emph{facet weights} for $(G,T)$. Due to Remark~\ref{rem:Gc}, a vector $c$ provides facet weights for $(G,T)$  if and only if $c$ has minimum integer form and there is a \emph{root basis} of $c$, i.e., a basis of $\R^{E(G)}$ consisting of roots of $c$.

A \emph{Steiner subgraph} of a Steiner graph $(G,T)$ is a Steiner graph $(G',T')$ with the same terminal set $T'=T$ and $G'$ being a subgraph of $G$.

%\begin{remark}
%The above remark implies that if inequality  $cx\ge d$ satisfies the second condition of the above, then it is a facet-inducing inequality  for $\cutd(G,T)$ whenever  $(G,T)$  contains $(G_c,T)$ as a subgraph.  (We assume here $c_e=0,e\in E(G)\setminus E_c)).$
%Therefore the study of $\cutd(G,T)$ is equivalent to the characterization of the subgraphs $(G',T)$ that admit  weight system $c: E(G')\rightarrow \Z_{>0}^{E(G')}$ such that $G'$ has $|E(G')|$ minimum weight Steiner cuts that are linearly independent.  We call these graphs $(G',T)$ \emph{facet-defining}, the weight system \emph{good}, and a \emph{root} is a Steiner cut whose weight is $\gamma(G_c,T)$ (i.e. is minimum).
%\end{remark}

\begin{remark}
\label{rem:reduction_to_facet_graphs}
    In order to determine an inequality description of $\cutd(G,T)$ for some Steiner graph $(G,T)$ it suffices due to Remark~\ref{rem:Gc}  to find all  facet weights for all facet inducing Steiner subgraphs of $(G,T)$.
\end{remark}

Therefore, we will concentrate on identifying facet inducing Steiner graphs and their facet weights subsequently. It will be one of the consequences of our results that for $|T|\le 5$ the facet weights of a facet inducing Steiner graph $(G,T)$ are uniquely determined. 
However, for larger values of $|T|$ 
this does not hold in general (see  Section~\ref{sec:conclusion}).

In the following remarks we collect a few useful observations on facet inducing Steiner graphs, where for a node $v\in V(G)$ of a graph $G$ we denote by $G \setminus v$ the graph $(V(G)\setminus\{v\},E(G)\setminus\delta(v))$.

%\begin{remark}\label{rem:facets_first_properties}
\ifINDICATECHANGES
    {\color{\changescolor}%
\fi
\begin{proposition}\label{prop:facets_first_properties}
 \added[comment={\revrem{2}{11}}]{}
 For every facet inducing Steiner graph  $(G,T)$  with facet weights $c$ the following hold:
	\begin{enumerate}
		\item $G$ is connected.
		\item Every node in $V(G)\setminus T$ has degree at least two.
		\item The facet of $\cutd(G,T)$ defined by  $cx\ge\gamma_c(G,T)$  is bounded.
		\item If $\delta(S)$ is  a root of $c$, then both $S$ and $V(G)\setminus S$ induce connected subgraphs of $G$.
		\item Every edge of $G$ is contained in at least one root of $c$.
		\item For every $v \in V(G)$, each component of $G \setminus v$ contains a terminal from $T$. 
		\item For every  $e \in E(G)$, we have $c(e) \le \gamma_c(G,T)$ with  equality  if and only if $e$ is a bridge (i.e., removing the edge $e$ from $G$ results in a disconnected graph).
	\end{enumerate}
\end{proposition}
\ifINDICATECHANGES
    }%
\fi
%\end{remark}

\ifINDICATECHANGES
    {\color{\changescolor}%
\fi
\begin{proof}
    \added[comment={\revrem{2}{11}}]{}%
    Let $F$ be the facet of $\cutd(G,T)$ defined by $cx\ge\gamma_c(G,T)$. 
    
    Due to $\gamma_c(G,T)>0$ all terminals are contained in the same  connected component of $G$. Suppose $\varnothing \ne W\subseteq V(G)$ and $E(W) \subseteq E(G)$ are  the node and edge  sets, respectively, of some connected component of $G$ with $W \cap T = \varnothing$.
    Due to $G_c=G$ we have $|W|\ge 2$, hence there is some $e_0 \in E(W)$. There cannot be a root $\delta(S)$  of $F$ with $e_0 \in \delta(S)$, because  for such a root $\delta(S)\setminus E(W)$ would be a $T$-Steiner cut with $c$-weight smaller than $c(\delta(S))=\gamma_c(G,T)$ (here, and also in some following arguments, we use that $c$ is all-positive). Thus $F$ is contained in the face defined by $x_{e_0}=0$, contradicting the fact that $F$ is a facet. Hence $G$ is connected.
    \added[comment={\revrem{1}{1}}]{The second statement follows similarly, as for a non-terminal node with only one incident edge 
    no root of $F$ could contain that edge.} 

    The third statement is due to the fact that the exreme rays of $\cutd(G,T)$ are contained in the nonnegative orthant (in fact, they are generated by the standard unit vectors). 

    Let $\delta(S)$ be a root of $c$.  If $S$ does not induce a connected subgraph of $G$ then there is a proper subset $S' \subsetneq S$ with $T\cap S' \ne \varnothing$ and $\delta(S') \subsetneq \delta(S)$, hence $\delta(S')$ is a $T$-Steiner cut with smaller $c$-weight than $c(\delta(S))=\gamma_c(G,T)$. If $V(G) \setminus S$ does not induce a conneted subgraph, a similar contradiction to the validity of $cx\ge\gamma_c(G,T)$ for $\cutd(G,T)$ is implied. Thus the fourth claim is established.

    The \replaced{fifth }{fivth} statement is clear since if there was an edge $e$ not contained in any root of $c$ than $F$ would be contained in the face defined by $x_e=0$.

    To prove the sixth claim assume that $v$ is a node such that $W \subseteq V(G) \setminus \{v\}$ with $W \cap T = \varnothing$ is the node set of a connected component of $G\setminus v$. Due to $|T| \ge 2$ we have $W \subsetneq V(G)\setminus \{v\}$, and due to the first statement, $G$ is connected. Thus there must be at last one edge connecting $v$ to $W$. Hence the set $E(W \cup \{v\})$ of edges of $G$ having both nodes in $W \cup \{v\}$ is not empty, again implying a contradiction to the validity of $cx \ge \gamma_c(G,T)$, because for every root $\delta(S)$ of $c$ we have that $\delta(S) \setminus E(W \cup\{v\})$ is a $T$-Steiner cut as well.

    Finally, let $e \in E(G)$ be some edge of $G$. According to the \replaced{fifth }{fivth} statement there is a $T$-Steiner cut $\delta(S)$ with 
    $c(\delta(S))=\gamma_c(G,T)$ and
    $e \in \delta(S)$. The inequality $c(e) \le \gamma_c(G,T)$ follows  readily from the nonnegativity of $c$, where in fact equality implies $\delta(S) = \{e\}$ (since $c$ is even all-positive) thus $e$ being a bridge. Conversely, let $e' = u'w'\in E(G)$ be a bridge of $G$ with $U' \ni u'$ and $W' \ni w'$ being the node sets of the two connected components into which the connected graph $G$ breaks up when removing the bridge $e'$. Applying the sixth statement to $u'$ and to $w'$ one finds $W' \cap T \ne \varnothing$ and $U' \cap T \ne \varnothing$, respectively. Hence, $\{e'\}=\delta(U')=\delta(W')$ is a root of $c$ implying $c(e') \ge \gamma_c(G,T)$, thus $c(e') = \gamma_c(G,T)$.
\end{proof}
\ifINDICATECHANGES
    }%
\fi

%\added[comment={\revrem{1}{1}}]{All these remarks are easy to prove. For instance, for the second one it suffices to observe that in case there is a non-terminal node of degree one with $e$ being its incident edge then due to $c_e>0$ no root contains $e$, implying that the face defined by $c$ is contained in the face defined by $x_e=0$.}

%IS THE FOLLOWING TRUE?}  Let $(G,T)$ be a facet-defining graph and $c$ be a good weight system. Then the minimum weight of a cut in $G$ is $\gamma_c(G,T)$. If true, it would imply that all facets have even rhs or rhs=1.

\section{Laminar families}
\label{sec:laminar}

For a subset  $A\subseteq V$ of a  ground set $V$ we denote by  $\bar{A}:=V\setminus A$ the \emph{complement}  of $A$. Two sets 
   $A$ and $B$ \emph{intersect}  if 
$A \cap B$, $A \setminus B$, and $B \setminus A $ are all nonempty. 

A family $\mathcal{L}$ of subsets of $V$ (i.e., a set of pairwise distinct subsets of $V$) is \emph{laminar} if no two sets in $\mathcal{L}$ intersect. 
 That is, $\mathcal{L}$ is laminar if and only if  each pair of sets is either disjoint or comparable w.r.t. inclusion. We define $\mathcal{L}_{\min}$ to be the subfamily consisting of the inclusionwise minimal members of $\mathcal{L}$. If $\mathcal{L}$ is a laminar family, then the members of   $\mathcal{L}_{\min}$  have pairwise empty intersections.

\begin{lemma}\label{lem:boundLaminarSets}
Every laminar family $\mathcal{L}$ of distinct nonempty  subsets of $V\ne\varnothing$ satisfies 
		$|\mathcal{L}|\le |V|+|\mathcal{L}_{\min}|-1$.
\end{lemma}

\begin{proof} 
Since the inequality in the lemma is equivalent to $|\mathcal{L}|-|\mathcal{L}_{\min}|\le |V|-1$, it suffices to establish it for a 
laminar family $\mathcal{L}$ of distinct nonempty subsets of $V$ that maximizes $|\mathcal{L}| - |\mathcal{L}_{\min}|$ and that has  $|\mathcal{L}|$  largest possible among all maximizers. 

It is easy to see that for such a family $|\mathcal{L}_{\min}|$ consists of all singletons of $V$. 
%We may assume that $\mathcal{L}_{\min}$ consists only of singleton elements in $V$, since if $v$ belongs to some $S \in \mathcal{L}_{\min}$ with $|S|\ge 2$, then adding $\{v\}$ to $\mathcal{L}$ and removing $S$ preserves the  assumptions of the lemma and leaves $|\mathcal{L}|$ and $|\mathcal{L}_{\min}|$  unchanged.
%In fact,  we can even assume that  $\mathcal{L}_{\min}$ consists of all the singleton elements in $V$. Indeed,  adding a singleton  $\{v\} \not\in\mathcal{L}$  to $\mathcal{L}$ preserves  laminarity.  According to our first assumption, $v$ is not contained in any set from $\mathcal{L}_{\min}$. Therefore $|\mathcal{L}|$ and $|\mathcal{L}_{\min}|$ are both increased by one when adding $\{v\}$ to $\mathcal{L}$.
Hence we have $|\mathcal{L}_{\min}|=|V|$, and the claim follows from the well-known (and easy to prove) fact that a laminar family of non-empty pairwise distinct subsets of $V$ cannot have more than $2|V|-1$ members.
\end{proof}

For the proof of Theorem~\ref{thm:laminarRootBasis} below (which follows arguments outlined  in~\cite{LauRaviSingh2011}) the following  two observations are  very useful. 

\begin{lemma}
    \label{lem:triple_intersection}
   If the sets $A$ and $B$ intersect, and the set $C$ intersects at least one of the sets 
  \[ A \cap B, \quad
        A \cup B, \quad
        A \setminus B, \quad
        B \setminus A\,,\]
    then $C$ and $A$  intersect or $C$ and $B$ intersect. 
\end{lemma}

\begin{proof}
    If $C$ and $A\cap B$ intersect, then we have
    $C \cap A \ne \varnothing$ and $C\cap B \ne \varnothing$, neither $A\subseteq C$ nor $B \subseteq C$, and 
    not both $C \subseteq A$ and $C \subseteq B$. 
    Hence $C$ and $A$ or $C$ and $B$ intersect.
    
    If $C$ and $A\cup B$ intersect, then neither $C \subseteq A$ nor $C \subseteq B$ holds. Due to the symmetry of the statement in the lemma, we may assume $C \cap A \ne \varnothing$. Thus, if   $A \not\subseteq C$ is true, $C$ and $A$ intersect. Otherwise (i.e., $A \subseteq C$ holds), $C\cap B \ne \varnothing$ is non-empty as well (since then, as $A$ and $B$ intersect, we  have $\varnothing \ne A\cap B \subseteq C$)  and $B\not\subseteq C$ holds (due to $A\cup B \not\subseteq C$), thus $C$ and $B$ intersect.
    
    As the statement of the lemma is symmetric in $A$ and $B$, it remains to consider the case that $C$ and $A \setminus B$ intersect. In this case we have $C\cap A \ne\varnothing$ and $A \not\subseteq C$. Hence, if $C$ and $A$ do not intersect, $C \subseteq A$ holds, which, however,  implies that $C$ and $B$ intersect since $C$ and $A \setminus B$ intersect. 
 \end{proof}

 For a family $\mathcal{L}$ of subsets of $V$ and some \replaced{$W \subseteq V$}{$S\subseteq V$}, we 
 define
 \begin{equation*}
     I(\replaced{W}{S},\mathcal{L}) = \{L \in \mathcal{L}\,:\, \replaced{W}{S} \text{ and } L \text{ intersect}\}.
 \end{equation*}
 \added[comment={\revrem{2}{14}}]{If $\mathcal{L}$ is laminar and $W$ is not contained in $\mathcal{L}$, then $I(W,\mathcal{L})$ is the subfamily that prevents $W$ from being added to $\mathcal{L}$ while maintaining laminarity.}

\begin{lemma}\label{lem:laminar}
\replaced{If $\mathcal{L}$ is a laminar family of subsets of $V$, and $L\in\mathcal{L}$ and $S\subseteq V$ intersect}{If $L \in \mathcal{L}$ is a member of a laminar family $\mathcal{L}$ of subsets of $V$, and $S\subseteq V$ is some set such that $S$ and $L$  intersect}, then  we have
	\begin{equation*}
		I(S \cap L,\mathcal{L}) , \quad
		I(S \cup L,\mathcal{L}) , \quad
		I(S \setminus L,\mathcal{L}) , \quad
		I(L \setminus S,\mathcal{L}) 
		\quad\subsetneq \quad
		I(S,\mathcal{L})\,.
	\end{equation*}
\end{lemma}

\begin{proof}
If $L' \in \mathcal{L}$ is a set which intersects one of the sets
\[
    S \cap L, \quad S \cup L, \quad S \setminus L, \quad L \setminus S\,,
\]
then, by Lemma~\ref{lem:triple_intersection}, $L'$ intersects $S$ or $L$, where the latter is impossible as both $L$ and $L'$ belong to the laminar family  $\mathcal{L}$ . This argument establishes the above  inclusions, which obviously are proper  
as  $L \in I(S,\mathcal{L})$ does not intersect any of the four sets above.
\end{proof}

\section{Laminar root bases}
\label{sec:root bases}

The aim of this section is to establish the  following result.

\begin{theorem}
\label{thm:laminarRootBasis}
   If $c$ provides facet weights for the facet inducing Steiner graph $(G,T)$ then there is a \emph{laminar root basis} for $c$, i.e., a root basis $\delta(S)$ ($S \in \mathcal{L}$) for $c$ with a laminar family $\mathcal{L}$ of  subsets of $V(G)$.
\end{theorem}

Before we start to prove Theorem~\ref{thm:laminarRootBasis}, we  have a closer look at laminar root bases. 

%\begin{remark} \label{rem:laminarRootBasis} 
\ifINDICATECHANGES
    {\color{\changescolor}%
\fi
\begin{proposition}\label{prop:laminarRootBasis}
 \added[comment={\revrem{2}{12}}]{}
Let  $\delta(S)$ ($S \in \mathcal{L}$) be a laminar root basis for facet weights $c$ of a facet inducing Steiner graph $(G,T)$ with a laminar family $\mathcal{L}$ of subsets of $V(G)$. Then 
$\mathcal{L}_{\min}$ consists of singleton elements that are in $T$. 
\end{proposition}
\ifINDICATECHANGES
    }%
\fi
%\end{remark}

\begin{proof}
Each set $S \in \mathcal{L}_{\min}$ must contain some  node from $T$, because $\delta(S)$ is a $T$-Steiner cut. If $S$ contains more than one node, then due to Part~(4) of Proposition~\ref{prop:facets_first_properties} there must be an edge with both end nodes in $S$. However, due to the laminarity of $\mathcal{L}$ such an edge is not contained in any of the cuts $\delta(S)$ ($S \in \mathcal{L}$), contradicting the fact that those cuts form a basis of $\R^{E(G)}$.
\end{proof}

In the following, in order to simplify  reading, we denote by  $\delta(S)$ also the  incidence vector of the subset $\delta(S)$ of the edge set of  a graph. 
	
\begin{lemma}\label{lem:uncrossing}
If $\delta(S_1)$ and $\delta(S_2)$ are roots of the facet weights $c$ for the facet inducing Steiner graph $(G,T)$ then  at least one of the following holds:

	\begin{center}
	    $\delta(S_1\cap S_2)$ and $\delta(S_1\cup S_2)$ are both roots and  
	\end{center}
	\begin{equation}
	    \label{eq:uncrossing_claim_capcup}
		\delta(S_1) + \delta(S_2)
		=
		\delta(S_1\cap S_2) + \delta(S_1\cup S_2)
	\end{equation}
	or
	\begin{center}
	    $\delta(S_1\setminus S_2)$ and $\delta(S_2\setminus S_1)$ are both roots and 
	\end{center}
	\begin{equation}
	    \label{eq:uncrossing_claim_symdiff}
		\delta(S_1) + \delta(S_2)
		=
		\delta(S_1\setminus S_2) + \delta(S_2\setminus S_1)\,.
	\end{equation}
\end{lemma}

\begin{proof}
\replaced{We define: }{We define} 
	\begin{equation*}
		S_{12}:= S_1 \cap S_2, \quad
		S_{\bar{1}\bar{2}} := V \setminus (S_1 \cup S_2), \quad
		S_{1\bar{2}}:= S_1 \setminus S_2, \quad
		S_{\bar{1}2} := S_2 \setminus S_1
	\end{equation*}
\deleted{(see Fig.~???).}
As both $\delta(S_1)$ and $\delta(S_2)$ are $T$-Steiner cuts, 
\deleted[comment={\revrem{1}{2}, \revrem{2}{14}}]{the set  $T$ is not contained in any row or column of Fig~???. Hence} 
at least one of  the following holds:
	\begin{equation}
	    \label{eq:uncrossing_capcup}
	    \text{$\delta(S_{12})$, $\delta(S_{\bar{1}\bar{2}})$ are both $T$-Steiner cuts }
	\end{equation}
	or
	\begin{equation}
	    \label{eq:uncrossing_symdiff}
	    \text{$\delta(S_{1\bar{2}})$, $\delta(S_{\bar{1}2})$ are both $T$-Steiner cuts}\,.
	\end{equation}
\added{%
Indeed, if $\delta(S_{12})$ is not a $T$-Steiner cut, then we have  $T\subseteq S_{12}$ or $T \cap S_{12} = \varnothing$, each of which implies $T \cap S_{\bar{1}2} \ne \varnothing$ and $T \cap S_{1\bar{2}} \ne \varnothing$, hence~\eqref{eq:uncrossing_symdiff}. The argument for the case that $\delta(S_{\bar{1}\bar{2}})$ is not a $T$-Steiner cut is similar. 
}
 
    If~\eqref{eq:uncrossing_capcup} holds, 
	 for the incidence vectors we use the relation
	\begin{equation}
	\label{eq:uncrossing_capcup_chi}
		\delta(S_1) + \delta(S_2)
		=
		\delta(S_{12}) + \delta(S_{\bar{1}\bar{2}})
		+2\cdot \delta(S_{1\bar{2}},S_{2\bar{1}})
	\end{equation}
	where $\delta(S_{1\bar{2}},S_{2\bar{1}})$ is the set of edges with one endnode in $S_{1\bar{2}}$ and the other one in $S_{2\bar{1}}$.
	Since both $\delta(S_1)$ and $\delta(S_2)$ are roots for the weights $c>\zerovec{}$ and~\eqref{eq:uncrossing_capcup} holds, this implies $\delta(S_{1\bar{2}},S_{2\bar{1}})=\emptyset$ and thus \eqref{eq:uncrossing_claim_capcup}.
	
	If \eqref{eq:uncrossing_symdiff} holds, a similar argument implies  \eqref{eq:uncrossing_claim_symdiff}.

%\begin{figure*}[ht]
%	\begin{center}
%		\input tikZ/uncross_bw
%	\end{center}
%	\caption{Notations used in Lemma~\ref{lem:uncrossing} with $\chi(\delta(S_1)) + \chi(\delta(S_2))$ in\-di\-ca\-ted. }
%\label{fig:uncrossing}
%\end{figure*}

\end{proof}

\paragraph{\bf Proof of Theorem~\ref{thm:laminarRootBasis}.}

Let  $(G,T)$ be a facet inducing Steiner graph with facet weights $c$. Let $\delta(S)$ ($S \in \mathcal{L}$) be a laminar family of linearly independent roots for $c$ that has a maximal number of members.

We claim that  $\mathcal{L}$ is a root basis.
   In order to establish this statement, assume that it does not hold. Then there are roots that are not contained in the subspace spanned by the members of  $\mathcal{L}$. Among those roots, we choose $\delta(S)$ to be one that minimizes $|I(S,\mathcal{L})|$. By maximality of $\mathcal{L}$ we have $I(S,\mathcal{L}) \ne \varnothing$.
    
    Let $L \in I(S,\mathcal{L})$. By Lemma~\ref{lem:uncrossing}, one of the following two cases applies.  
    \begin{description}
        \item[Case 1: ]$\delta(S\cap L)$, $\delta(S \cup L)$ are roots and
 	    \begin{equation*}
		   \delta(S) + \delta(L)
		    =
		    \delta(S \cap L) + \delta(S \cup L)\,.
	    \end{equation*}
	    As $\delta(S)$ is not contained in the subspace spanned by $\mathcal{L}$, this equation shows that at least one of $\delta(S\cap L)$ and $\delta(S\cup L)$ is not contained in that subspace. Because of Lemma~\ref{lem:laminar}, this, however, contradicts the choice of $S$ (minimizing  $|I(S,\mathcal{L})|$). 
        \item[Case 2: ]$\delta(S\setminus L)$,  $\delta(L\setminus S)$ are roots and
 	    \begin{equation*}
		    \delta(S) + \delta(L)
		    =
		   \delta(S\setminus L) + \delta(L\setminus S)\,,
	    \end{equation*}
	    which implies a contradiction that is similar to the one we encountered in Case~1.
\end{description}
This completes the proof of Theorem~\ref{thm:laminarRootBasis}.
\added[comment={\revrem{2}{15}}]{\hfill\qed}

\replaced[comment={\revrem{1}{3}}]{%
For a laminar family $\mathcal{L}$ of subsets of $V(G)$ with a Steiner graph $(G,T)$ the following holds: all cuts induced by $\mathcal{L}$ are $T$-Steiner cuts if and only if each set in $\mathcal{L}_{\min}$ contains a node from $T$ and in case of $\mathcal{L}$ containing only a single maximal set w.r.t. inclusion that maximal set does not contain the entire terminal set $T$. Thus, defining the \emph{width} of a laminar family $\mathcal{L}$ to be $|\mathcal{L}_{\min}|$ if $\mathcal{L}$ contains more than one maximal set w.r.t. inclusion and to be $|\mathcal{L}_{\min}|+1$ otherwise, Theorem~\ref{thm:laminarRootBasis} and Proposition~\ref{prop:laminarRootBasis} imply that the Steiner \replaced[comment={\revrem{2}{16}}]{degree }{rank} of a facet defining inequality for a cut dominant is the smallest width of any laminar root basis for that inequality. 
}{%
For a laminar family $\mathcal{L}$  we define the \emph{width} of $\mathcal{L}$ to be $|\mathcal{L}_{\min}|$ unless $\mathcal{L}$ has only one maximal member w.r.t. inclusion, in which case the width is $|\mathcal{L}_{\min}|+1$. 
Theorem~\ref{thm:laminarRootBasis} then in particular implies that the Steiner rank of a facet defining inequality for a cut dominant is the smallest width of any laminar root basis for that inequality.%
}

\iffalse

\begin{corollary}
\label{cor:edgeNumberEssential}
 Every dense facet-defining $(G,T)$-inequality $cx\ge\gamma$ satisfies 
 \begin{equation}\label{eq:mu}
    |E(G)| \le |V(G)| + |T| -3\,.
 \end{equation}
If equality holds in \eqref{eq:mu}, then we have
\[
	c(\delta(\{t\}\})) = \gamma\quad\text{for all }t \in T\,.
\]
%and every pair $v,w \in V(G)$ (with $v \ne w$) is separated by a root cut of $cx\ge\gamma$.
 \end{corollary}

\begin{proof}
Every  dense facet-defining $(G,T)$-inequality has a laminar root basis $\mathcal{L}$ (due to Thm.~\ref{thm:laminarRootBasis}) which clearly satisfies $|E(G)| = |\mathcal{L}|$. According to Lem.~\ref{lem:laminar_root_bases}, $\mathcal{L}$  is complement-free with $\varnothing \not\in \mathcal{L}$ and $\cup\mathcal{L}_{\min} \subseteq T$. Therefore,  Lem.~\ref{lem:boundLaminarSets} 
implies the claims in the corollary.
\end{proof}

\fi

\section{Subdivision and gluing}
\label{sec:split_glue}

In this section we  investigate two   operations that construct facet inducing Steiner graphs from smaller ones. We also characterize  operations that are inverse to those operations.

 We say that the Steiner graph $(G,T)$ has been obtained by a \emph{subdivision} from the Steiner graph $(G',T')$ if $G$ arises from $G'$ by  replacing an edge $uv$ with the edges $uw$ and $vw$, where $w \not\in V(T')$ is a newly added node, and $T=T'$ holds (i.e., the new node $w \not\in T$ is a non-terminal node). If $c'$ is a vector of edge weights of $G'$ then 
 \added[comment={\revrem{2}{17}}]{we define }the accordingly subdivided vector  $c$ of edge weights of $G$ \replaced{via }{has} $c(uw)=c(vw)=c'(uv)$ and $c(e)=c'(e)$ for all other edges. 
 \deleted[comment={\revrem{2}{19}}]{The following fact is straight forward.}

%\begin{remark}
 \ifINDICATECHANGES
    {\color{\changescolor}%
\fi
\begin{proposition}	\label{prop:subdivisions_preserve_facets}  
    If the Steiner graph $(G,T)$ is \replaced[comment={\revrem{2}{18}}]{obtained }{obained} by subdivision from a facet inducing Steiner graph $(G',T')$ with facet weights $c'$,  then also $(G,T)$ is a facet inducing Steiner graph with facet weights $c$ obtained by subdividing $c'$ accordingly. 
\end{proposition}
 \ifINDICATECHANGES
    }%
\fi
%\end{rem}

\begin{proof}
    We continue to use the notations introduced when defining the subdivision operation. Due to $w \not\in T$, for every $T$-Steiner cut $\delta_{G}(S)$ in $(G,T)$ we have that $\delta_G(S')$ with $S' := S\cap V(G')$ is a $T'$-Steiner cut in $(G',T')$ with 
    \begin{equation*}
        c(\delta_{G}(S)) 
        = 
        \begin{cases}
            c(\delta_{G'}(S')) + 2c'(uv) & \text{if} \{uw,vw\} \subseteq \delta(S) \\
            c(\delta_{G'}(S')) & \text{otherwise}
        \end{cases}.
    \end{equation*}
    In particular, we conclude $\gamma_{c}(G,T)\ge\gamma_{c'}(G',T')$.  
    Conversely, each $T'$-Steiner cut $\delta_{G'}(S')$ in $(G',T')$ with $uv \not\in \delta_{G'}(S')$  gives rise to one $T$-Steiner cut $\delta_G(S) = \delta_{G'}(S')$  with $c(\delta(S)) = c'(\delta_{G'}(S'))$, and each $T'$-Steiner cut $\delta_{G'}(S')$ in $(G',T')$ with $uv \in \delta_{G'}(S')$ induces two $T$-Steiner cuts $\delta_G(S_u) = \delta_{G'}(S') \cup\{uw\}$ and $\delta_G(S_v) = \delta_{G'}(S') \cup \{vw\}$ with $c(\delta_G(S_u)) = c(\delta_G(S_v)) = c'(\delta_{G'}(S'))$. In particular, we have $\gamma_{c}(G,T)\le\gamma_{c'}(G',T')$, hence $\gamma_{c}(G,T)=\gamma_{c'}(G',T')$. Moreover, if $\mathcal{S}'$ is a root base for $c'$, then with choosing $S^{\star} \in \mathcal{S}'$ arbitrarily one finds that
    \begin{equation*}
        \{\delta_G(S_u)\,:\, S'\in\mathcal{S}'\}
        \cup
        \{\delta_G(S^{\star}_v)\}
    \end{equation*}
    forms a root basis for $c$.
    \end{proof}
 
% \begin{proof}
%     We continue to use the notations introduced when defining the subdivision operation. 

%     The statement follows observing that the obvious map that assigns to every $T'$-Steiner cut $\delta_{G'}(S')$ in $G'$ the $T$-Steiner 
%     \begin{equation*}
%         \delta(S) := \delta(S') \cap E(G) \cup
%         \begin{cases}
%             \{uv\} & \text{if $|\delta_{G'}(S') \cap \delta_{G'}(w)| = 1$} \\
%             \varnothing & \text{otherwise}
%         \end{cases}
%     \end{equation*}
% \end{proof}

In order to investigate the operation that is reverse to subdivision, let 
$w\in V(G)$ be a node of degree two in the graph $G$ whose two neighbors  $u$ and $v$ are not adjacent to each other. Then \emph{reducing} $w$ means to remove $w$ and to replace its two incident edges  $uw$ and $vw$  by the edge $uv$.

Analyzing the effect of reductions on facet inducing Steiner graphs is a bit more involved than the reasoning behind Proposition~\ref{prop:subdivisions_preserve_facets} is. 
The following simple observation will be  helpful for that purpose. 

%\begin{remark}
 \ifINDICATECHANGES
    {\color{\changescolor}%
\fi
\begin{lemma}\label{lem:delta_nonterm}
\added[comment={\revrem{2}{19}}]{}
If $\delta(S)$ is a root of the weight vector $c$ for a Steiner graph $(G,T)$ and $w\in V(G) \setminus T$ is a non-terminal node, then $c(\delta(w)\cap\delta(S)) \le c(\delta(w)\setminus\delta(S))$ holds. 
\end{lemma}
 \ifINDICATECHANGES
    }%
\fi
%\end{remark}

\begin{proof}
We may assume $w \in S$. The claim then follows because $\delta(S')$ with $S' = S \setminus \{w\}$ is a $T$-Steiner cut as well with 
\begin{equation*}
    c(\delta(S)) = \gamma_c(G,T) \le c(\delta(S')) = c(\delta(S)) + c(\delta(w)\setminus\delta(S)) - c(\delta(w)\cap\delta(S)).
\end{equation*}
\end{proof}%

\begin{theorem}
	\label{thm:degree_two_nonterminals_come_from_subdividing}
	Let $(G,T)$ be a facet inducing Steiner graph with facet weights $c$.  Let $w\in V(G)\setminus T$ be a non-terminal node with  exactly two neighbors, say $u$ and $v$. Then the following hold:
	\begin{enumerate}
	    \item We have $c(uw)=c(vw)$ and $u$ and $v$ are not adjacent.
	    \item If $G'$ is obtained from $G$ by  reducing  node $w$, then $(G',T')$ with $T'=T$ is a facet inducing Steiner graph with facet weights $c'$ obtained from $c$ by setting $c'(uv)=c(uw)=c(vw)$ and $c'(e)=c(e)$ for all other edges. 
	\end{enumerate}
	Note that $(G,T)$ can be obtained from $(G',T')$ by  subdivision of the edge $uv$, and $c$ is the weight vector obtained by subdividing $c'$ accordingly.
\end{theorem}

\begin{proof}   
Due to $w\not\in T$ (and $c > \zerovec{}$) no root of $c$ contains both $uw$ and $vw$
\added[comment={\revrem{1}{4}}]{%
(because in that case removing those two edges from the root would result in a $T$-Steiner cut with smaller $c$-value)%
}%
. A first consequence is that (since each of those two edges is contained in at least one root of~$c$) Lemma~\ref{lem:delta_nonterm}  implies 
$c(uw)=c(vw)$. As a second consequence, $u$ and $v$ are not adjacent to each other, because otherwise (as every cut intersects a triangle in none or two edges) all roots of $c$ 
\replaced[comment={\revrem{1}{5}, \revrem{2}{20}}]{are }{were} 
contained in the hyperplane defined by $x_{uw}+x_{vw}=x_{uv}$, a contradiction to the existence of a root basis for $c$.

For the proof of the second statement we first observe that if $\delta_{G'}(S')$ with  $S'\cap\{u,v\}\ne\varnothing$ is a $T$-Steiner cut  in $G'$, then $\delta_G(S'\cup\{w\})$ is a $T$-Steiner cut in $G$ with $c'(\delta_{G'}(S'))=c(\delta_G(S'\cup\{w\}))$. This implies $\gamma_{c'}(G',T')\ge\gamma_c(G,T)$. 
In fact, we have $\gamma_{c'}(G',T')=\gamma_c(G,T)$, since for each root $\delta_G(S)$ of $c$ with $w\in S$ we have $S\cap\{u,v\}\ne\varnothing$, thus $c'(\delta_{G'}(S'))=c(\delta_G(S))$ holds for the $T$-Steiner cut $\delta_{G'}(S')$ in $G'$ with $S' = S \setminus\{w\}$, which hence is a root of $c'$. 

In order to show that there is a root basis for $c'$ let $\delta_G(S)$ ($S \in \mathcal{S}$) be a root basis for $c$ with $w\in S$ for all $S \in \mathcal{S}$. As we just argued above, for each $S \in \mathcal{S}$ the $T$-Steiner cut $\delta_{G'}(S')$  with $S'=S\setminus\{w\}$ is a root for $c'$. Since those roots are the images of the root basis $\delta_G(S)$ ($S \in \mathcal{S}$) under the surjective linear map $\R^{E(G)}\rightarrow\R^{E(G')}$ that maps $x$ to $x'$ with $x'_{uv}=x_{uw}+x_{vw}$ and $x'_e = x_e$ for all other edges, they  contain a root basis for $c'$.
\end{proof}

The second operation we consider is to \emph{glue} two Steiner graphs $(G_1,T_1)$ and $(G_2,T_2)$ with $V(G_1)\cap V(G_2) = \{w\}$ and $w \in T_1\cap T_2$
\deleted{(possibly after replacing the graphs by isomorphic copies)} at a common terminal  node $w$ in order to obtain a Steiner graph $(G,T)$ with $V(G)=V(G_1)\cup V(G_2)$ and $E(G)=E(G_1)\cup E(G_2)$ as well as $T=T_1\cup T_2$ or $T=(T_1\cup T_2) \setminus \{w\}$, i.e. $w$ may or may not belong to $T$.
For two  vectors $c_1$ and $c_2$ of edge weights of $G_1$ and $G_2$ in minimum integer form, respectively, the vector obtained by gluing $c_1$ and $c_2$ is the  vector $c$ of edge weights of $G$ that is the minimum integer form of 
\[
    (\gamma_{c_2}(G_2,T_2)\cdot c_1,\gamma_{c_1}(G_1,T_1)\cdot c_2)\,.
\]
The set of roots of $c$ is the union of the set of roots of $c_1$ and the set of roots of $c_2$, and  
$\gamma_c(G,T)$ 
\replaced[comment={\revrem{1}{6}}]{divides }{devides} 
the least common multiple of $\gamma_{c_1}(G_1,T_1)$ and $\gamma_{c_2}(G_2,T_2)$. Consequently, the following holds.

%\begin{remark}
 \ifINDICATECHANGES
    {\color{\changescolor}%
\fi
\begin{proposition}\label{prop:identifying_x}
 If the Steiner graph $(G,T)$ is obtained by gluing  two facet inducing Steiner graphs $(G_1,T_1)$ and $(G_2,T_2)$ with facet weigts $c_1$ and $c_2$, respectively,  then $(G,T)$ is a facet inducing Steiner graph as well and the  weights  obtained from gluing $c_1$ and $c_2$ accordingly are facet weights for $(G,T)$. 
 %FURTHERMORE ALL FACET WEIGHTS OF $(G,T)$ ARISE THIS WAY.
 \end{proposition}
 \ifINDICATECHANGES
    }%
\fi
%\end{remark}

For the analysis of the reverse operation to gluing recall that a 
\emph{cut node} in a connected graph $G$ is a node $w$ for which $G\setminus w$ is not connected.

\begin{theorem} 
\label{thm:cut_nodes}
Let $(G,T)$ be a facet inducing Steiner graph with facet weights $c$, and let  $w\in V(G)$ be a cut node of $G$. Let $V_1\subseteq V(G)$ be the node set of one of the  components of $G\setminus w$ and $V_2 := V(G\setminus w)\setminus V_1$. We denote by $G_1$ and $G_2$  the subgraphs of $G$ induced by $V_1\cup\{w\}$ and $V_2\cup\{w\}$, respectively, and define $T_1= (T\cap V_1) \cup\{w\}$ as well as $T_2= (T\cap V_2) \cup\{w\}$. 

Then we have $|T_1|,|T_2| \ge 2$, and both $(G_1,T_1)$ and $(G_2,T_2)$ are facet inducing Steiner graphs with facet weights $c_1$ and $c_2$, respectively, that are the minimum integer forms of the restrictions of $c$ to $E(G_1)$ and to $E(G_2)$, respectively.

Note that $(G,T)$ can be obtained by gluing $(G_1,T_1)$ and $(G_2,T_2)$ at their common terminal node $w$, and $c$ is the weight vector obtained by gluing  $c_1$ and $c_2$ accordingly. 
\end{theorem}

\begin{proof}
    The fact that we have $|T_1|,|T_2| \ge 2$  is due to Part~6 of Proposition~\ref{prop:facets_first_properties}. 
    
    Let $\delta_G(S)$ ($S \in \mathcal{S}$) be a root basis for $c$ with $w \not\in S$ for all $S\in\mathcal{S}$.  
    Part~4 of Proposition~\ref{prop:facets_first_properties} implies that  for every $S \in \mathcal{S}$ we have $S\subseteq V_1$ or $S \subseteq V_2$. From this we conclude that $\delta_{G_1}(S)$ ($S \in \mathcal{S}$, $S \subseteq V_1$) form a root basis of $c_1$, and $\delta_{G_2}(S)$ ($S \in \mathcal{S}$, $S \subseteq V_2$) form a root basis of $c_2$.
\end{proof}

\section{Irreducible facet inducing Steiner graphs}
\label{sec:main:I}

We define a  (weighted) Steiner graph $(G,T)$ to be  \emph{constructable} from a family $\mathcal{F}$ of  (weigh\-ted) Steiner graphs if there is a sequence $(G_0,T_0)$, \dots, $(G_r,T_r)=(G,T)$ with $(G_0,T_0)\in\mathcal{F}$ such that, for each $i \in [r]$, the Steiner graph  $(G_i,T_i)$ can be obtained as a subdivision of $(G_{i-1},T_{i-1})$ or from gluing $(G_{i-1},T_{i-1})$  with some member of $\mathcal{F}$. 
%The class of Steiner graphs that are constructible from $\mathcal{F}$ is the smallest class of Steiner graphs that contains $\mathcal{F}$ and is closed under the operations of subdividing and gluing.
Due to Proposition~\ref{prop:subdivisions_preserve_facets} and Proposition~\ref{prop:identifying_x} every (weighted) Steiner graph that is constructable from a family of facet inducing (weighted) Steiner graphs is facet inducing. 

We call a facet inducing Steiner graph $(G,T)$ \emph{irreducible}  if $G$ has no cut node and no node in  $V(G)\setminus T$ has degree two. 
The results from Section~\ref{sec:split_glue} imply the following.

\begin{remark}
\label{rem:reduction_to_irreducibles}
A (weighted) Steiner graph $(G,T)$ is  facet inducing  if and only if it is constructable from irreducible facet inducing (weighted) Steiner graphs 
\replaced[comment={\revrem{1}{7}}]{(each of which then  has at most $|T|$ terminals). }{with at most $|T|$ terminals.}
\end{remark}

Therefore, it suffices to determine the (weighted) irreducible facet inducing Steiner graphs. Due to~\eqref{eq:st_cut_dom} the only irreducible facet inducing Steiner graph with two terminals  is the \emph{Steiner edge} $(K_2,V(K_2))$ and the facet inducing Steiner graphs with two terminals are the \emph{Steiner paths} $(G,\{s,t\})$, where $G$ is a path with end nodes $s$ and $t$. 

\begin{remark}
\label{rem:degrees_irreducible}
Each irreducible facet inducing Steiner graph  $(G,T)$ with $|T|\ge 3$ has $\deg(v)\ge 2$ for all $v\in T$ and $\deg(v)\ge 3$ for all $v \in V(G)\setminus T$.
\end{remark}

We are now prepared to establish a bound that in particular  implies the first main result stated in the introduction.

\begin{lemma}	
\label{lem:bound}  
Every irreducible facet inducing Steiner graph $(G,T)$  with $|T|\ge 3$ has 
\begin{equation*}
   |E(G)|\le |V(G)|+|T|-3
   \quad\text{and}\quad
   |V(G)|\le 3|T|-6\,.
 \end{equation*}
   If equality holds in any of those two inequalities, then, for every choice of  facet weights,  $\delta(t)$ is a root for every $t \in T$. 
\end{lemma}

\begin{proof} 
\deleted[comment=\revrem{1}{8}]{Due to Remark~\ref{rem:degrees_irreducible} we have
    $2|E(G)|\ge 2|T|+3|V(G)\setminus T|=3|V(G)|-|T|$.
}
 
 Let  $c$ be facet weights of $(G,T)$ and let $\delta(S)$ ($S\in \mathcal{L}$) be a laminar root basis for $c$ as guaranteed to exist by Theorem~\ref{thm:laminarRootBasis}. Due to Proposition~\ref{prop:laminarRootBasis}  there is some terminal $t^{\star}\in T$ with $\{t^{\star}\} \in \mathcal{L}_{\min}$. 
 Replacing every set $S\in\mathcal{L}$ with $t^{\star} \in S$ by its complement $V(G)\setminus S$, we may assume that $\mathcal{L}$ is a laminar family (note that those sets $S$ form a maximal chain in $\mathcal{L}$, thus laminarity is preserved) of pairwise distinct non-empty subsets of $V(G)\setminus\{t^{\star}\}$. 
  Again by Proposition~\ref{prop:laminarRootBasis}, we have $|\mathcal{L}_{\min}|\le |T|-1$. Therefore, the bound from 
 Lemma~\ref{lem:boundLaminarSets} yields
 \begin{equation*}
     |E(G)| = |\mathcal{L}| \le |V(G)\setminus\{t^{\star}\}|+|\mathcal{L}_{\min}|-1\le |V(G)|+|T|-3\,,
 \end{equation*}
  where equality between the left-hand and the right-hand side implies $|\mathcal{L}_{\min}|=|T|-1$\added[comment={\revrem{2}{22}}]{ (see   the second inequality)}, i.e. that  $\mathcal{L}_{\min}$ contains each terminal-singleton from $T\setminus\{t^{\star}\}$. 

 \added[comment=\revrem{1}{8}]{On the other hand, due to Remark~\ref{rem:degrees_irreducible} we have}
 \added{
 \[
    2|E(G)|\ge 2|T|+3|V(G)\setminus T|=3|V(G)|-|T|\,.
\]
}
Combining the above two inequalities we obtain $|V(G)|\le 3|T|-6$ as claimed.
\end{proof}

The above lemma has the following immediate consequence. 

\begin{theorem}\label{thm:bounded_number_terminals_irreducible} 
There is a function  $f$ such that the number of 
 irreducible facet inducing Steiner graphs $(G,T)$ is bounded by $f(|T|)$ (up to isomorphisms).
\end{theorem}

\deleted{Of course, each irreducible facet inducing Steiner graph admits only finitely many facet weights (since a polyhedron has only finitely many facets).} 
In particular, Theorem~\ref{thm:bounded_number_terminals_irreducible}    yields the next result.

\begin{theorem}
\label{thm:bounded_rhs}
There is a function  $g$ such that the facet defining inequalities (in minimum integer form) for every Steiner cut dominant $\cutd(G,T)$ have right-hand sides (and coefficients) that are bounded by $g(|T|)$. 
\end{theorem}

\begin{proof}
 \replaced[comment={\revrem{2}{24}}]{%
 For each $\tau \in \{2,3,\dots\}$ let us denote by $\mathcal{F}(\tau)$ the set of (isomorphism classes of) irreducible facet inducing Steiner graphs with at most $\tau$ terminals, which is a finite set 
 according to Theorem~\ref{thm:bounded_number_terminals_irreducible}. As every  polyhedron has only finitely many facets, the sets
\begin{equation*}
    \mathcal{F}^{\star}(\tau) = \{((G,T),c)\,:\,(G,T)\in\mathcal{F}(\tau), c \text{ facet weights for $(G,T)$}\}
\end{equation*}
 are finite as well. Since a subdivision operation does not change the right-hand side and a gluing operation produces an inequality (in minimum integer form) whose right-hand side divides the least common multiple of the right-hand sides of the two inequalities that have been glued together 
 (see the statement preceding Proposition~\ref{prop:identifying_x}), 
 the right-hand side (in minimum integer form) of every inequality that can be constructed from $\mathcal{F}^{\star}(\tau)$ divides the least 
 common multiple of the  right-hand sides of the finitely many 
 inequalities  represented by $\mathcal{F}^{\star}(\tau)$. Due to 
 Remarks~\ref{rem:reduction_to_facet_graphs} and~\ref{rem:reduction_to_irreducibles} this proves the theorem.
}{%
This follows from Theorem~\ref{thm:bounded_number_terminals_irreducible} 
via   
Remark~\ref{rem:reduction_to_facet_graphs}
and 
Remark~\ref{rem:reduction_to_irreducibles}, 
 since the right hand-side of every inequality that is constructable from a finite set of inequalities (in minimum integer form) %\replaced[comment={\revrem{2}{23}}]{divides}{ devides} 
 divides the least common multiple of their right-hand-sides.
}
\end{proof}

\section{Steiner trees and Steiner cacti}
\label{sec:trees_and_cacti}

We call a Steiner graph 
$(G,T)$ a  \emph{Steiner tree} if $G$ is a  tree  whose degree-one nodes are all in $T$ (see Figure~\ref{fig:SteinerTree}). 
The Steiner trees are the Steiner graphs that are constructible from  Steiner edges. In particular they are facet inducing, where assigning a weight of one to each edge yields the unique facet weight vector (in minimum integer form).  
We call the corresponding inequalities (with right-hand-side one) \emph{Steiner tree inequalities}.

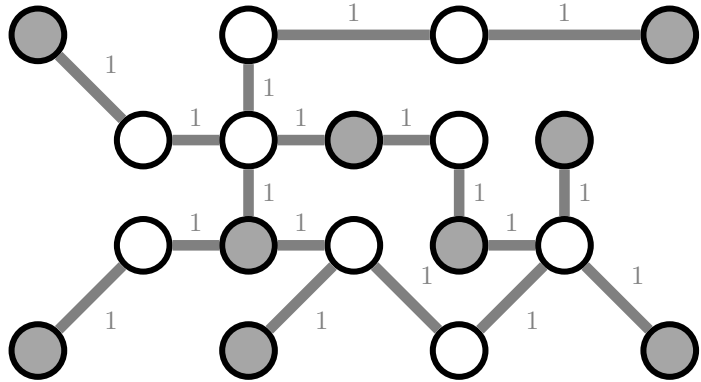
\begin{figure}[ht]
    \renewcommand{\scale}{.7}
    \centering
    \begin{tikzpicture}[scale=\scale]
%,every edge quotes/.style={fill=white,font=\footnotesize}]]

%    \draw[step=.2cm,gray,very thin] (-2,-2) grid (2,2);
        
    \node [term] (v1) at (0,6) {};
    \node [nonterm] (v2) at (4,6) {};
    \node [nonterm] (v3) at (8,6) {};
    \node [term] (v4) at (12,6) {};
    \node [nonterm] (v5) at (2,4) {};
    \node [nonterm] (v6) at (4,4) {};
    \node [term] (v7) at (6,4) {};
    \node [nonterm] (v8) at (8,4) {};
    \node [term] (v9) at (10,4) {};
    \node [nonterm] (v10) at (2,2) {};
    \node [term] (v11) at (4,2) {};
    \node [nonterm] (v12) at (6,2) {};
    \node [term] (v13) at (8,2) {};
    \node [nonterm] (v14) at (10,2) {};
    \node [term] (v15) at (0,0) {};
    \node [term] (v16) at (4,0) {};
    \node [term] (v17) at (8,0) {};
    \node [term] (v18) at (12,0) {};
    
    \draw [edge1] (v2) edge["$1$"] (v3);
    \draw [edge1] (v3) edge["$1$"] (v4);
    \draw [edge1] (v5) edge["$1$"] (v6);
    \draw [edge1] (v6) edge["$1$"] (v7);
    \draw [edge1] (v7) edge["$1$"] (v8);
    \draw [edge1] (v10) edge["$1$"] (v11);
    \draw [edge1] (v11) edge["$1$"] (v12);
    \draw [edge1] (v13) edge["$1$"] (v14);
    \draw [edge1] (v1) edge["$1$"] (v5);
    \draw [edge1] (v2) edge["$1$"] (v6);
    \draw [edge1] (v6) edge["$1$"] (v11);
    \draw [edge1] (v8) edge["$1$"] (v13);
    \draw [edge1] (v9) edge["$1$"] (v14);
    \draw [edge1] (v10) edge["$1$"] (v15);
    \draw [edge1] (v12) edge["$1$"] (v16);
    \draw [edge1] (v12) edge["$1$"] (v17);
%    \draw [edge1] (v14) edge["$1$"] (v17);
    \draw [edge1] (v14) edge["$1$"] (v18);
\end{tikzpicture}
    \caption{A Steiner tree with its facet weights (the dark nodes are the terminals).}
    \label{fig:SteinerTree}
\end{figure}

Since for every Steiner graph $(G,T)$, a subset of $E(G)$ contains a $T$-Steiner cut if and only if it intersects $E(G')$ for each Steiner subtree $(G',T)$ of $(G,T)$ the Steiner tree inequalities associated with the Steiner subtrees of $(G,T)$ (together with the nonnegativity constraints) provide an integer programming formulation for $\cutd(G,T)\cap\Z^{E(G)}$.

Conversely, if for a Steiner subgraph $(G',T)$ of the Steiner graph $(G,T)$ the edge set $E(G')$ intersects every $T$-Steiner cut in $(G,T)$ then $(G',T)$ contains a Steiner subtree. This implies the following.

\begin{remark}\label{rem:rhs_one}
    The only facet defining inequalities (in minimum integer form) with right-hand-side one for Steiner cut dominants are the Steiner tree inequalities.  
\end{remark}

A \emph{cactus} is a connected graph (we only consider graphs without loops or multiple edges) that has at 
\replaced[comment={\revrem{1}{9}, \revrem{2}{27}}]{least }{meast} one cycle, but in which every edge is contained in at most one cycle.
 We call the Steiner graph $(G,T)$ a \emph{Steiner cactus} if $G$ is a cactus  whose degree-one nodes are in $T$, and in which every cycle contains at least three nodes that are cut nodes of $G$ or terminals (this in particular implies  $|T|\ge 3$) (see Figure~\ref{fig:SteinerCactus}).

% The most simple Steiner cacti are \emph{Steiner cycles}, i.e.,  Steiner graphs $(G,T)$ with $T=V(G)$ where $G$ is a cycle (thus $|V(G)|\ge 3$). Steiner cycles are irreducible facet inducing Steiner graphs for which the unique facet weights are provided by the all-ones vector and the right-hand side equals two. 
 
 Steiner cacti are the Steiner graphs that are constructible from  Steiner edges and 
 \replaced[comment={\revrem{1}{10}, \revrem{2}{28}}]{\emph{Steiner cycles}, i.e., Steiner graphs $(G,V(G))$ with a cycle $G$}{ Steiner cycles}. In particular, they are facet inducing, where assigning to each edge weight  one or two  depending on whether the edge is contained in some cycle or not, respectively, 
 yields the unique facet weight vector \added{$c$} (in minimum integer form)\added{, where $\gamma_c=2$}. 
We call the corresponding inequalities with right-hand-side two \emph{Steiner cacti inequalities}. 
 \added[comment={\revrem{1}{11}}]{Here, the weights of value two arise since when gluing a Steiner cactus inequality with right-hand-side two and a Steiner edge inequality with right-hand-side one the coefficient of the latter one enters scaled by a factor of two.}

\begin{figure}[ht]
    \renewcommand{\scale}{.7}
    \centering
    \begin{tikzpicture}[scale=\scale]
%,every edge quotes/.style={fill=white,font=\footnotesize}]]

%    \draw[step=.2cm,gray,very thin] (-2,-2) grid (2,2);
        
    \node [term] (v1) at (0,6) {};
    \node [nonterm] (v2) at (4,6) {};
    \node [nonterm] (v3) at (8,6) {};
    \node [term] (v4) at (12,6) {};
    \node [nonterm] (v5) at (2,4) {};
    \node [nonterm] (v6) at (4,4) {};
    \node [term] (v7) at (6,4) {};
    \node [nonterm] (v8) at (8,4) {};
    \node [term] (v9) at (10,4) {};
    \node [nonterm] (v10) at (2,2) {};
    \node [nonterm] (v11) at (4,2) {};
    \node [term] (v12) at (6,2) {};
    \node [term] (v13) at (8,2) {};
    \node [nonterm] (v14) at (10,2) {};
    \node [term] (v15) at (0,0) {};
    \node [term] (v16) at (4,0) {};
    \node [term] (v17) at (8,0) {};
    \node [nonterm] (v18) at (12,0) {};
    
    \draw [edge1] (v2) edge["$1$"] (v3);
    \draw [edge2] (v3) edge["$2$"] (v4);
    \draw [edge1] (v5) edge["$1$"] (v6);
    \draw [edge1] (v6) edge["$1$"] (v7);
    \draw [edge1] (v3) edge["$1$"] (v8);
    \draw [edge1] (v4) edge["$1$"] (v9);
    \draw [edge1] (v7) edge["$1$"] (v8);
    \draw [edge1] (v5) edge["$1$"] (v10);
    \draw [edge1] (v10) edge["$1$"] (v11);
    \draw [edge2] (v11) edge["$2$"] (v12);
    \draw [edge1] (v13) edge["$1$"] (v14);
    \draw [edge2] (v1) edge["$2$"] (v5);
    \draw [edge1] (v2) edge["$1$"] (v6);
    \draw [edge1] (v6) edge["$1$"] (v11);
    \draw [edge1] (v9) edge["$1$"] (v14);
    \draw [edge2] (v10) edge["$2$"] (v15);
    \draw [edge2] (v11) edge["$2$"] (v16);
    \draw [edge1] (v13) edge["$1$"] (v17);
    \draw [edge1] (v14) edge["$1$"] (v17);
    \draw [edge1] (v14) edge["$1$"] (v18);
    \draw [edge1] (v4) edge["$1$"] (v18);
\end{tikzpicture}
    \caption{A Steiner cactus with its facet weights (the dark nodes are the terminals).}
    \label{fig:SteinerCactus}
\end{figure}
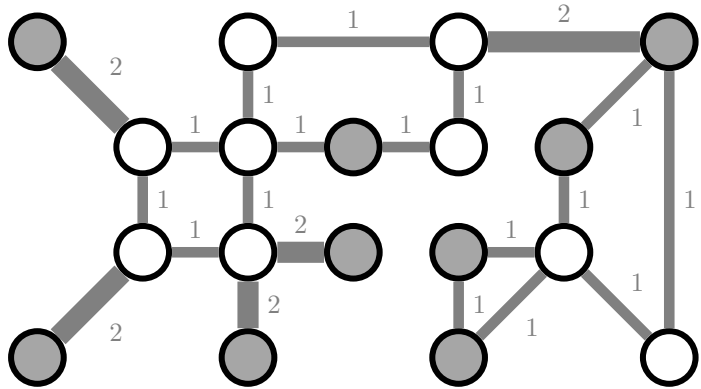

 It seems worth to note that in contrast to Remark~\ref{rem:rhs_one} 
 there are facet defining inequalities (in minimum integer form) with right-hand-side two  for Steiner cut dominants that are not Steiner cacti inequalities. Indeed, whenever $G$ arises from some two-edge \replaced{connected }{conected} graph by replacing each edge by a path consisting of three edges, then $(G,T)$ with $T\subseteq V(G)$ including all nodes of $G$ with degree two is an irreducible facet inducing Steiner graph with facet weights provided by the all-ones vector and right-hand side two. \added[comment={\revrem{2}{29}}]{See Figure~\ref{fig:three_paths} for an example.}
 
\begin{figure}[ht]
    \renewcommand{\scale}{.7}
    \centering
    \begin{tikzpicture}[scale=\scale]
%,every edge quotes/.style={fill=white,font=\footnotesize}]]

%    \draw[step=.2cm,gray,very thin] (-2,-2) grid (2,2);
        
    \node [nonterm] (v1) at (4,6) {};
    \node [term] (v2) at (0,4) {};
    \node [term] (v3) at (4,4) {};
    \node [term] (v4) at (8,4) {};
    \node [term] (v5) at (0,2) {};
    \node [term] (v6) at (4,2) {};
    \node [term] (v7) at (8,2) {};
    \node [nonterm] (v8) at (4,0) {};

    \draw [edge1] (v1) edge["$1$"] (v2);
    \draw [edge1] (v1) edge["$1$"] (v3);
    \draw [edge1] (v1) edge["$1$"] (v4);
    \draw [edge1] (v2) edge["$1$"] (v5);
    \draw [edge1] (v3) edge["$1$"] (v6);
    \draw [edge1] (v4) edge["$1$"] (v7);
    \draw [edge1] (v8) edge["$1$"] (v5);
    \draw [edge1] (v8) edge["$1$"] (v6);
    \draw [edge1] (v8) edge["$1$"] (v7);
\end{tikzpicture}
    \caption{An example of a facet inducing Steiner graph that is not a Steiner cactus with facet weights in minimum integer form and right-hand-side two. A root basis is formed by the six cuts defined by the single terminals and the three cuts defined by the pairs of adjacent terminals.}
    \label{fig:three_paths}
\end{figure}

As it simplifies later arguments, we state the following fact (that follows easily from Theorems~\ref{thm:degree_two_nonterminals_come_from_subdividing} and ~\ref{thm:cut_nodes}).
\begin{remark}\label{rem:treecactus_Steiner}
\replaced{Let $(G,T)$ be a Steiner graph with $G$ being a cactus. Then $(G,T)$ is facet inducing if and only if $(G,T)$ is a Steiner cactus; in this case it follows from Theorem~\ref{thm:cut_nodes} and the previous discussion that the unique facet weights are as described in  the definition of a Steiner cactus. }{If $(G,T)$ is a facet inducing Steiner graph where $G$ is a tree or a cactus, then $(G,T)$ is a Steiner tree or a Steiner cactus, respectively.}
\end{remark}

Since the only connected subgraphs of trees are trees, we find that if $(G,T)$ is a Steiner graph with a tree $G$, then $\cutd(G,T)$ is described by the nonnegativity constraints and the Steiner tree inequality  defined by the unique Steiner subtree  of $(G,T)$ with terminal set $T$.  

Similarly, the only connected subgraphs of cacti are trees and cacti. Hence if $(G,T)$ is a Steiner graph with a cactus $G$, then $\cutd(G,T)$ is described by the nonnegativity constraints and the Steiner tree and Steiner cactus inequalities defined by the Steiner subtrees and by the Steiner subcacti of $(G,T)$ with terminal set $T$, respectively.

\section{At most five terminals}
\label{sec:main:II}

The facet inducing Steiner graphs with two terminals are the Steiner paths, among which the irreducible ones are the Steiner edges. 
The purpose of this section is to prove  the following result.

\begin{theorem}
	\label{thm:facets_three_four_five}
	The facet inducing Steiner graphs with at most five terminals are   Steiner trees and  Steiner cacti. 
\end{theorem}

Theorem~\ref{thm:facets_three_four_five} implies that, for $\tau\in\{3,4,5\}$, the irreducible facet inducing Steiner graphs with $\tau$ terminals are cycles of length $\tau$ with all nodes being terminals.
\iffalse
Before we prove that theorem, we state two ONE? consequences. OMIT?

\begin{corollary}
	\label{cor:E_three_four_five}
		For $\tau\in\{3,4,5\}$, the irreducible facet inducing Steiner graphs with $\tau$ terminals are the terminal cycles of length $\tau$.
\end{corollary}
\fi
In order to prove Theorem~\ref{thm:facets_three_four_five} it remains to show that every facet inducing Steiner graph $(G,T)$ with $|T|\in\{3,4,5\}$ is a Steiner tree or a Steiner cactus. 
\added[comment=We removed Lemma~6.]{}
The following operation turns out to be very helpful in reducing the sizes of graphs  that we need to consider.

\begin{lemma}\label{le-Y} 
Let $(G,T)$ be a facet inducing Steiner graph with facet weights $c$ and  $v\in V(G)\setminus T$ a non-terminal node of degree three with neighbors $v_1$, $v_2$, and $v_3$. With  $e_i:=vv_i$ for all $i\in \{1,2,3\}$ the following statements hold:

\begin{enumerate}
\item No root contains $\{e_1, e_2,e_3\}$.
\item We have 
\begin{align*}
    c(e_1) & \le c(e_2) + c(e_3) \\  
    c(e_2) & \le c(e_3) + c(e_1) \\  
    c(e_3) & \le c(e_1) + c(e_2)  \,, 
\end{align*}
where at most one of the inequalities is satisfied at equality.
\item 
If $c(e_i)< c(e_j)+c(e_k)$ holds for $i,j,k$ with  $\{i,j,k\}=\{1,2,3\}$ then no root contains both $e_j$ and $e_k$, 
\replaced[comment={\revrem{2}{32}}]{ which further implies that}{ and} 
$v_jv_k$ is not an edge of $G$.
\end{enumerate}
\end{lemma}
\begin{proof} 
\replaced[comment={\revrem{1}{12}}]{Due to $v \in V(G) \setminus T$, for every root taking its symmetric difference with $\delta(v)$ results in another $T$-Steiner cut and therefore does not decrease the $c$-value. This immediately implies the first statement, and the second one follows as every edge is contained in at least one root (see Proposition~\ref{prop:facets_first_properties}). }%
{The first statement follows from the fact that $v\in V(G)\setminus T$, which also implies the inequalities in the second statement (as every edge is contained in at least one root, see Proposition~\ref{prop:facets_first_properties}).%
}
Clearly, because of $c>\zerovec{}$, no two of those inequalities can be satisfied at equality. 

To prove the third statement, observe that in case of  $c(e_i)< c(e_j)+c(e_k)$ (for $\{i,j,k\}=\{1,2,3\}$) no root contains both $e_j$ and $e_k$ because of $v\in V(G)\setminus T$. Consequently, $v_jv_k \in E(G)$ in this case would  imply that every root $x$ satisfies $x_{v_jv_k}=x_{e_j}+x_{e_k}$, which, however, contradicts the fact that the roots span $\R^{E(G)}$.
\end{proof}

\begin{definition}
	\label{def:YD}
	Let $(G,T)$ be a 
 \replaced[comment=It seems more natural not to restrict the definition.]{ }{facet inducing} Steiner graph with facet weights $c$ and  $v\in V(G)\setminus T$ a non-terminal node of degree three with neighbors $v_1$, $v_2$, and $v_3$. For each $i \in \{1,2,3\}$, we denote by  $e_i:=vv_i\in E(G)$ the edge in $G$ connecting $v$ to $v_i$ and by $f_i$ the edge (possibly in $E(G)$ or not) connecting the two neighbors of $v$ different from $v_i$, and define $\zeta_i := c(\delta_G(v))-2c(e_i)$.
Let $G'$ be the graph with
	$
		V(G') = V(G) \setminus\{v\}
	$
	and
 \added[comment={\revrem{2}{33}}]{}
    \begin{equation*}    
        E(G') = \added{(}E(G) \setminus \{e_1,e_2,e_3\}\added{)} 
        \cup  \setdef{f_i}{\zeta_i > 0}\,.
    \end{equation*}
	Let \(\replaced{c^{\star}}{c'}\in\R^{E(G')}\) be the vector obtained from $c$ by removing the components indexed by $e_1$, \(e_2\), and \(e_3\), and setting
%	\[
%		c'(f_i) := 
%		\begin{cases}
%			c(f_i) & \text{if }f_i \in E(G)\\
%			\frac12\zeta_i                & \text{otherwise}
%		\end{cases}
%	\]
    \[
        \replaced{c^{\star}}{c'}(f_i) := c(f_i) + \frac{\zeta_i}{2}  
        \added{\quad \in \frac12 \Z}
    \]
	for all $i \in \{1,2,3\}$ with $f_i \in E(G')$ (with $c(f_i):=0$ in case of $f_i \not\in E(G)$).
 \added[comment={\revrem{1}{13}}]{We denote by $c'$ the minimum integer form of $c^{\star}$, i.e., $c'$ equals $c^{\star}$ or $2c^{\star}$ depending on whether $c^{\star}$ has a fractional component or not.}
%	\begin{eqnarray*}
%		c'_{f_1} & := & \frac{1}{2}(c_{e_2}+c_{e_3}-c_{e_1})  \quad (+ c_{f_1} \text{ if }f_1 \in E(G))	\\	
%		c'_{f_2} & := & \frac{1}{2}(c_{e_3}+c_{e_1}-c_{e_2}) \quad (+ c_{f_2} \text{ if }f_2 \in E(G))	\\	
%		c'_{f_3} & := & \frac{1}{2}(c_{e_1}+c_{e_2}-c_{e_3}) \quad (+ c_{f_3} \text{ if }f_3 \in E(G))\,.	
%	\end{eqnarray*}
%	\[
%		c'_{f_i} := 
%		\begin{cases}
%			\frac{1}{2}(c(\delta_G(v))-c_{e_i}) + c_{f_i} & \text{if }f_i \in E(G) \\
%			\frac{1}{2}(c(\delta_G(v))-c_{e_i})  & \text{otherwise}	
%		\end{cases}
%	\]
%	for all \(i\in\{1,2,3\}\). 

	Then we say that the Steiner graph $(G',T')$ with $T'=T$ and  $c'$ have  been obtained from $(G,T)$ and $c$ by applying the \emph{\YDred\  at $v$} (see Figure~\ref{fig:YDred}).
\end{definition}
\begin{figure}
    \centering
    \begin{tikzpicture}

%    \draw[step=.2cm,gray,very thin] (-2,-2) grid (2,2);
        
    \node [pp_node] (v) at (0,0) {$v$};
    \node [pp_node] (v1) at (30:2) {$v_1$};
    \node [pp_node] (v2) at (150:2) {$v_2$};        
    \node [pp_node] (v3) at (270:2) {$v_3$};

	\draw [pp_edge] (v) -- node [above, near start] {$e_1$} (v1) ;
	\draw [pp_edge] (v) -- node [above, near start] {$e_2$} (v2);
	\draw [pp_edge] (v) -- node [right, near start] {$e_3$} (v3);        
	\draw [pp_edge] (v1) -- node [above] {$f_3$} (v2);
	\draw [pp_edge] (v2) -- node [below = 5pt] {$f_1$} (v3);
	\draw [pp_edge] (v3) -- node [below = 5pt] {$f_2$} (v1);        
\end{tikzpicture}
    \caption{Notations used for \YDred s.}
    \label{fig:YDred}
\end{figure}

\deleted[comment=Appeared to be unnecessary.]{Note that applying a \YDred\ as in Definition~\ref{def:YD} to a simple graph (recall that we restrict our considerations to simple graphs throughout the paper) results in a simple graph again. }
Lemma~\ref{le-Y} 
\deleted{now} implies the following.

\begin{remark}
	\label{rem:zeta123}
\replaced[comment=Now we need to restrict explicitly to facet inducing Steiner graphs.]{If one applies a \YDred\ to a facet inducing Steiner graph $(G,T)$, then using the notations introduced in Definition \ref{def:YD}}{In Definition \ref{def:YD},}  we have
	\begin{equation}
		\label{eq:zeta123:nonneg}
		\zeta_i \ge 0 \quad\text{for all }i \in \{1,2,3\}
	\end{equation}
	with 
	\begin{equation}
	\label{eq:zeta123:equality}
	\zeta_i = 0 \quad\text{for at most one }i \in \{1,2,3\}\,.
    \end{equation}
    Furthermore, if 	$\zeta_i >0$ then $f_i \not\in E(G)$ holds. \added[comment={\revrem{2}{34}}]{(See Figure~\ref{fig:YDconf}.)}
\end{remark}

\begin{figure}
    \centering
    \begin{tikzpicture}

%    \draw[step=.2cm,gray,very thin] (-2,-2) grid (2,2);
        
    \node [pp_node] (v) at (0,0) {$v$};
    \node [pp_node] (vi) at (30:2) {$v_i$};
    \node [pp_node] (vj) at (150:2) {$v_j$};        
    \node [pp_node] (vk) at (270:2) {$v_k$};

	\draw [pp_edge] (v) -- node [above, near start] {$e_i$} (vi) ;
	\draw [pp_edge] (v) -- node [above, near start] {$e_j$} (vj);
	\draw [pp_edge] (v) -- node [right, near start] {$e_k$} (vk);        
\end{tikzpicture}
    \hspace{1cm}
    \begin{tikzpicture}

%    \draw[step=.2cm,gray,very thin] (-2,-2) grid (2,2);
        
    \node [pp_node] (v) at (0,0) {$v$};
    \node [pp_node] (vi) at (30:2) {$v_i$};
    \node [pp_node] (vj) at (150:2) {$v_j$};        
    \node [pp_node] (vk) at (270:2) {$v_k$};

	\draw [pp_edge] (v) -- node [above, near start] {$e_i$} (vi) ;
	\draw [pp_edge] (v) -- node [above, near start] {$e_j$} (vj);
	\draw [pp_edge] (v) -- node [right, near start] {$e_k$} (vk);        
	\draw [pp_edge] (vi) -- node [above] {$f_k$} (vj);
\end{tikzpicture}    
    \caption{The two possibilities for the subgraph of $G$ induced by the non-terminal node $v$ and its three neighbors when a \YDred\ is applied.} 
    \label{fig:YDconf}
\end{figure}

\begin{lemma}
	\label{lem:YD_facet}
	Applying a \YDred\ to a facet inducing Steiner graph $(G,T)$ with facet weights $c$ results in a facet inducing Steiner graph $(G',T')$ with $T'=T$ and  facet weights $c'$
 %\added[comment={\revrem{1}{13}}]{or $2c'$  depending on whether $c'$ is integral or it has at least one fractional (half-integral) component}
 \added{(as defined in Definition~\ref{def:YD}) with $\gamma_{c'}(G',T') = \gamma_c(G,T)$ or $\gamma_{c'}(G',T')=2\gamma_{c}(G,T)$}.
\end{lemma}

\begin{proof} 
%Let $\delta_G(S)$ be a Steiner cut of $(G,T)$ where $v\not\in S$.  If $S$ contains $v_1,v_2,v_3$, then  $c(\delta_G(S))>c(\delta_G(S\cup\{v\}))$. If $S$ contains, say $v_2,v_3$, by Lemma \ref{le-Y} we have $c(\delta_G(S))\ge c(\delta_G(S\cup\{v_1\}\setminus \{v_2,v_3\}))$. Since $c(e_i)=\zeta_j+\zeta_k$, then $\gamma_{c'}(G',T')=\gamma_c(G,T)$ and $c(\delta_G(S))=c'(\delta_{G'}(S))$ for every root $\delta(S)$.
\added{We continue to use the notations of Definition~\ref{def:YD}. Recall that we have $c'=c^{\star}$ or $c'=2c^{\star}$, in particular $\gamma_{c'}(G',T') = \gamma_{c^{\star}}(G',T')$ or $\gamma_{c'}(G',T')=2\gamma_{c^{\star}}(G',T')$. } We first observe that  $\replaced{c^{\star}}{c'}(f)>0$ holds for all $f \in E(G')$  (see~\eqref{eq:zeta123:nonneg}) . 

With $\mathcal{S}'$ denoting the family of all  $S\subseteq V(G)\setminus \{v\}$ with
$S \cap T \ne \varnothing$, $T \not\subseteq S$, and
$|S \cap \{v_1,v_2,v_3\}| \le 1$, we have 
\begin{equation}\label{YNabla-weights}
    \replaced{c^{\star}}{c'}(\delta_{G'}(S)) = c(\delta_G(S)\added{)}
    \quad
    \text{for all } S \in \mathcal{S}'\,.
\end{equation}
Indeed, this clearly holds for all $S \in \mathcal{S}'$ with $S \cap \{v_1,v_2,v_3\}=\varnothing$, and if $S \in \mathcal{S}'$ satisfies, say, $S \cap \{v_1,v_2,v_3\}=\{v_1\}$ we find
\begin{equation*}
     \replaced{c^{\star}}{c'}(\delta_{G'}(S)) 
    = c(\delta_G(S)) - c(e_1)  + \frac{\zeta_2}{2} + \frac{\zeta_3}{2}
    = c_G(\delta_G(S)\added{)}
\end{equation*}
as well. 

As the family $\delta_{G'}(S)$ ($S \in \mathcal{S}'$) is the family of all $T'$-Steiner cuts in $G'$, and since $\delta_G(S)$ is a $T$-Steiner cut in $G$ for each $T'$-Steiner cut $\delta_{G'}(S)$ in $G'$ with $S \in \mathcal{S}'$ (recall $T'=T$), the equations in~\eqref{YNabla-weights} imply
 $\gamma_{\replaced{c^{\star}}{c'}}(G',T') \ge \gamma_c(G,T)$.

In order to prove that $c'$ \added{(which equals $c^{\star}$ or $2c^{\star}$)} provide facet weights, 
let the cuts $\delta_G(S)$ for $S \in \mathcal{S}$ form a root basis  of $c$ with $v \not\in S$ for all $S \in \mathcal{S}$. For each $S \in \mathcal{S}$, the cut  $\delta_{G'}(S)$ clearly is a $T'$-Steiner cut in $G'$, for which we claim 
\begin{equation}\label{YNablaRootsRoots}
\replaced{c^{\star}}{c'}_{G'}(\delta_{G'}(S)) = c_G(\delta(S) = \gamma_c(G,T)\,,
\end{equation}
where this in particular implies $\gamma_{\replaced{c^{\star}}{c'}}(G',T') = \gamma_c(G,T)$\,.
Indeed, due to Lemma~\ref{le-Y} each $S \in \mathcal{S}$ satisfies $|S\cap\{v_1,v_2,v_3\}| \le 2$. If that intersection has less than two elements, then we have $S \in \mathcal{S}'$ (with $\mathcal{S}'$ defined above), thus \eqref{YNablaRootsRoots} is one of the equations~\eqref{YNabla-weights}. Otherwise, say if $S\cap\{v_1,v_2,v_3\} = \{v_2,v_3\}$ holds, we have $\zeta_1 = 0$ due to Lemma~\ref{le-Y}, and hence
\begin{align*}
     \replaced{c^{\star}}{c'}(\delta_{G'}(S)) 
    &= c(\delta_G(S)) - c(e_2) - c(e_3)  + \frac{\zeta_2}{2} + \frac{\zeta_3}{2}\\
    &= c_G(\delta_G(S)) - \zeta_1\\
    &= c_G(\delta_G(S))
\end{align*}
as well. 

We thus have proved that the cuts $\delta_{G'}(S)$ ($S \in \mathcal{S}$) are roots of \added{$c^{\star}$, and hence roots of} $c'$, and therefore it suffices to show that they span  $\R^{E(G')}$. 
	 Towards this end, let  $\varphi:\R^{E(G)}\rightarrow\R^{E(G')}$ be the  linear map with  
$
\varphi(x)_e = x_e 
$
for all $e \in E(G') \cap E(G)$
and
$\varphi(x)_{f_i} = x(\delta_G(v))-x_{e_i}$ for all $i \in \{1,2,3\}$ with $f_i \in E(G') \setminus E(G)$. As the image of $\varphi$ is the entire space $\R^{E(G')}$, it hence is enough to argue that for each $S \in \mathcal{S}$ we have $\varphi(\delta_G(S)) = \delta_{G'}(S)$. As we clearly have $\varphi(\delta_G(S))_{e} = \delta_{G'}(S)_{e}$ for every $e \in E(G') \cap E(G)$, we thus only have to establish  $\varphi(\delta_G(S))_{f_i} = \delta_{G'}(S)_{f_i}$ for each $i \in \{1,2,3\}$ with $f_i \in E(G') \setminus E(G)$. For such an index $i$ and $\{1,2,3\}=\{i,j,k\}$ we have $\varphi(\delta_G(S))_{f_i}=|\delta_G(S) \cap \{e_j,e_k\}|$ and $\zeta_i > 0$, where according to Lemma~\ref{le-Y} the latter inequality  implies 
$|\delta_G(S) \cap \{e_j,e_k\}| \le 1$. Due to   $|S\cap\{v_j,v_k\}|=|\delta_G(S) \cap \{e_j,e_k\}|$, and as we have $\delta_{G'}(S)_{f_i}=1$ if and only if $|S \cap \{v_j,v_k\}| = 1$ holds, this completes the proof.
 %Let $\delta_G(S)$ ($S \in \mathcal{S}$) be a root basis  of $c$ with $v \not\in S$ for all $S \in \mathcal{S}$. 
%	We show that the cuts $\delta_{G'}(S)$ ($S \in \mathcal{S}$) span $\R^{E(G')}$. 
%	 Towards this end, let  $\varphi:\R^{E(G)}\rightarrow\R^{E(G')}$ be the  linear map with  
%$
%\varphi(x)_e = x_e 
%$
%for all $e \in E(G') \cap E(G)$
%and
%$\varphi(x)_{f_i} = x(\delta_G(v))-x_{e_i}$ for all $i \in \{1,2,3\}$ with $f_i \in E(G')$. As the image of $\varphi$ is the entire space $\R^{E(G')}$, by the above argument,  $\varphi(\delta_G(S)) = \delta_{G'}(S)$ for all $S \in \mathcal{S}$.
    \end{proof}

It is easy to see that applying a \YDred\ (according to Definition~\ref{def:YD}) to a Steiner tree  or to a Steiner cactus yields a Steiner cactus. 
\added[comment={\revrem{1}{13}}]{Indeed, relying on the notations introduced in Definition~\ref{def:YD}, if $v$  is not contained in any cycle then (both in the tree as well as in the cactus case) the star $e_1,e_2,e_3$ is replaced by the triangle $f_1,f_2,f_3$ of new edges, and if $v$ is contained in a (then unique) cycle containing, say,  $e_1$ and $e_2$, then we have $\zeta_1 = \zeta_2 = 1$ and $\zeta_3 = 0$, and therefore  in that cycle $e_1,e_2$ are replaced by  $f_1,f_2$ ($e_3$ is removed, $f_3$ is not newly introduced, but was already present in case that cycle is a triangle).} 
We will, however, \added{also} need the following weakened version of the reverse statement that requires some \added{more involved} arguments to be established. 

\begin{lemma}
	\label{lem:YD_no_tree_or_cactus}
	If $(G',T')$ is a Steiner tree or a Steiner cactus that has been obtained from a facet inducing Steiner graph $(G,T)$ by applying a \YDred\ (according to Definition~\ref{def:YD}), then $(G,T)$ is a Steiner tree or a Steiner cactus.
\end{lemma}

\begin{proof}
	We continue to use the notations from Definition~\ref{def:YD} and furthermore denote by $G_0$ the subgraph of $G$ induced by $\{v,v_1,v_2,v_3\}$. 
	Due to~\eqref{eq:zeta123:equality}, after possibly renumbering the nodes, we have
			$\zeta_1,\zeta_2 > 0$. This readily implies $f_1,f_2 \in E(G')$ and (see Remark~\ref{rem:zeta123}) $f_1,f_2 \not\in E(G)$. 
\iffalse
	\begin{equation*}
		f_1 \in E(G')
		\quad\text{and}\quad
		f_2 \in E(G')\,.
	\end{equation*}
	In order to see that it also implies
	\begin{equation}
	\label{eq:f_f_not_in_EG}
		f_1 \not\in E(G)
		\quad\text{and}\quad
		f_2 \not\in E(G)\,,
	\end{equation}
 let us assume 
		$f_1 \in E(G)$
	(the arguments for the case $f_2\in E(G)$ are similar).  Due to $\zeta_1>0$ we find that no root $\delta_G(S)$ of $c$ contains both $e_2$ and $e_3$, 
		since in that case (we can assume $v\in S$) the $T$-Steiner cut $\delta_G(S \setminus\{v\})$ yields 
	\[
		\gamma_c(G,T)\le c(\delta_G(S\setminus\{v\})) \le c(\delta_G(S))-\zeta_1 = \gamma_c(G,T)-\zeta_1 < \gamma_c(G,T)\,.
	\]
	But if  each root of $c$ contains at most one of the edges $e_2,e_3$ then (as cuts intersect the triangle  $\{f_1,e_2,e_3\}$ in none or two edges) they all satisfy the equation 
	 $x_{f_1}=x_{e_2}+x_{e_3}$, which again yields a contradiction. Thus~\eqref{eq:f_f_not_in_EG} is established. 
\fi
    \added[comment={\revrem{2}{36}}]{We subdivide the remainder of the proof into two cases.}

    \paragraph*{\added{Case~1: $f_3\in E(G')$}\ \\}
	\deleted{Let us first consider the case $f_3\in E(G')$, i.e., }
    \added{In this case, } the entire triangle $\{f_1,f_2,f_3\}$ is contained in  $E(G')$. \replaced{Hence }{Then} $G'$ (which is a tree or a cactus) is a cactus, and the graph obtained from $G'$ by removing the three edges $f_1,f_2,f_3$ 
	has three  connected components
	 $G_1,G_2,G_3$, where each $G_i$ intersects  the triangle  exactly in node $v_i$. 
	As $G'$ is a cactus, every graph $G_i$ is a tree (possibly consisting of the single node $v_i$) or a cactus. Due to $f_1,f_2\not\in E(G)$ the graph $G_0$ 
	is either the star with edge set 
	$
	E(G_0)=\{e_1,e_2,e_3\}
	$
	 or the cactus with edge set 
	 $
	 	E(G_0)=\{e_1,e_2,e_3,f_3\}$.
	Therefore, 
	$
		G=G_0 \cup G_1 \cup G_2 \cup G_3
	$
	is a tree or a cactus.

    \paragraph*{\added{Case~2: $f_3\not\in E(G')$}\ \\}
    \deleted{It remains to consider the case $f_3\not\in E(G')$.} 
    In this case, by definition of $G'$, we also have $f_3\not\in E(G)$. Thus $G$ is the graph that arises from $G'$ by removing $f_1$ and $f_2$ and adding the star $G_0$ with edge set 
	$
	E(G_0)=\{e_1,e_2,e_3\}$.
	
	If at most one of the edges $f_1,f_2$ is contained in some cycle of $G'$ then $G$ is connected with no edge being contained in more than one cycle, hence $G$ is a tree or a cactus.
	
	Therefore it remains to consider the case that $f_1$ is contained in some cycle $C_1\subseteq E(G')$ and $f_2$ is contained in some cycle $C_2\subseteq E(G')$ of the cactus $G'$. 
	As $G'$ is a cactus, the two cycles are either identical  or the intersection of their node sets is $\{v_3\}$. 
	
	If $C_1=C_2$ holds, then $G$ is a cactus (arising from the cactus $G'$ by inserting the star $G_0$ ``into that cycle'' and removing $f_1,f_2$). 

	Otherwise (i.e., the two cycles $C_1$ and $C_2$ with $f_1\in C_1$ and $f_2\in C_2$ intersect only in the common node $v_3$ of $f_1$ and $f_2$)
 \replaced{%
 the graph $G$ contains the cycle $C_1\cup C_2 \setminus\{f_2,f_1\}\cup\{e_1,e_2\}$ and its chord $e_3$. This implies that each root of $c$ that contains $e_3$ (and such a root exists due to Proposition~\ref{prop:facets_first_properties}) contains at least three edges, which, however, (due to $c(e) \ge 1$ for all $e \in E(G)$) contradicts  $\gamma_c(G,T) \le \gamma_{c'}(G',T') \le 2$, where the first inequality follows from Lemma~\ref{lem:YD_facet} and the second one is due to $(G',T')$ being a Steiner tree or a Steiner cactus.
 }{%
	let $H$ be the cactus that arises from the cactus $G'$ by removing $f_1,f_2$ and adding the node $v$ as well as the two edges $e_1,e_2$. The cactus $H$ contains the cycle 
    $C:=C_1\cup C_2 \setminus\{f_2,f_1\}\cup\{e_1,e_2\}$ 
 with node set $U$. The graph $G$ arises from adding the chord $e_3$ of $C$ to the cactus $H$. 
 % Hence the subgraph $G''=(U,C\cup\{e_3\})$ is a two-connected component of $G$. Thus, if we denote by 
 Denoting by $G''$ the subgraph $(U,C\cup\{e_3\})$ of $G$ and by $T''$ the union of $U\cap T$ and the set of all  nodes in $U$ that are cut nodes of $G$, we conclude from Theorem~\ref{thm:cut_nodes} that $(G'',T'')$ is a facet inducing Steiner graph. 
 However, $G''$ is a Hamiltonian cycle (with edge set $C$) plus the one edge $e_3$, which contradicts Lemma~\ref{lem:Hamiltonian_cycle}.} 
\end{proof}

Finally,  we will make use of  the main result of Conforti, Fiorini, and Pashkovich~\cite{ConfortiFioriniPashkovich2016}, which translated into our terminology is the following.

\begin{theorem}[Thm. 5 and Lem.~4 in~\cite{ConfortiFioriniPashkovich2016}]
	\label{thm:CFP}
	If $cx\ge\gamma$ (in minimum integer form) defines a facet of $\cutd(G)$   with $\gamma > 2$ then $G$ has a minor that is a prism or a pyramid (see Figure~\ref{fig:prismPyramid}).
\end{theorem}

\begin{figure}[ht]\label{fig:prismPyramid}
	\centerline{%
		\begin{tikzpicture}

%    \draw[step=.2cm,gray,very thin] (-2,-2) grid (2,2);
        
    \node [pp_node] (v1) at (90:2) {};
    \node [pp_node] (v2) at (210:2) {};        
    \node [pp_node] (v3) at (330:2) {};
    \node [pp_node] (w1) at (90:.7) {};
	\node [pp_node] (w2) at (210:.7) {};        
	\node [pp_node] (w3) at (330:.7) {};
		
	\draw [pp_edge] (v1) -- (v2);
	\draw [pp_edge] (v2) -- (v3);
	\draw [pp_edge] (v3) -- (v1);        
	\draw [pp_edge] (w1) -- (w2);
	\draw [pp_edge] (w2) -- (w3);
	\draw [pp_edge] (w3) -- (w1);        
	\draw [pp_edge] (v1) -- (w1);
	\draw [pp_edge] (v2) -- (w2);
	\draw [pp_edge] (v3) -- (w3);        
\end{tikzpicture}
		\hspace{1cm}
		\begin{tikzpicture}

%    \draw[step=.2cm,gray,very thin] (0,0) grid (5,4);

    \node [pp_node] (v1) at (90:2) {};
	\node [pp_node] (v2) at (210:2) {};        
	\node [pp_node] (v3) at (330:2) {};
	\node [pp_node] (w1) at (90:1) {};
	\node [pp_node] (w2) at (210:1) {};        
	\node [pp_node] (w3) at (330:1) {};
	\node [pp_node] (u) at (0:0) {};	

	\draw [pp_edge] (v1) -- (v2);
	\draw [pp_edge] (v2) -- (v3);
	\draw [pp_edge] (v3) -- (v1);        
	\draw [pp_edge] (v1) -- (w1);
	\draw [pp_edge] (v2) -- (w2);
	\draw [pp_edge] (v3) -- (w3);        
	\draw [pp_edge] (u) -- (w1);        
	\draw [pp_edge] (u) -- (w2);        
	\draw [pp_edge] (u) -- (w3);        
\end{tikzpicture}	
	}
	\caption{The \emph{prism} (left) and the \emph{pyramid} (right).}
\end{figure}
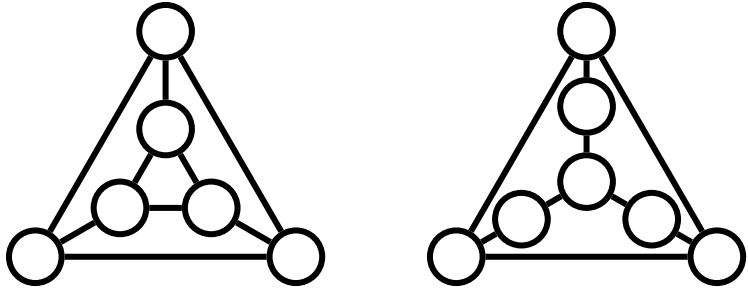

\begin{corollary}
	\label{cor:CFP}
	If $(G,T)$ is a facet inducing Steiner graph with facet weights $c$ that satisfy $\gamma:=\gamma_c(G,T)>2$ and  $cx\ge\gamma$ is valid for $\cutd(G)$, then  $G$ has a minor that is a prism or a pyramid.	
\end{corollary}

\begin{proof}
	This follows readily from Theorem~\ref{thm:CFP} as 
	due to
	\[
		\dim(\cutd(G)) = |E(G)| = \dim(\cutd(G,T)) 
	\]
	every inequality that is facet defining for $\cutd(G,T)$ and   valid for $\cutd(G)$ is also facet defining for $\cutd(G)$ (which contains $\cutd(G,T)$).
\end{proof}

In order to establish Theorem~\ref{thm:facets_three_four_five}, it suffices to prove the following result (see Remark~\ref{rem:treecactus_Steiner}).

\begin{proposition}
	\label{prop:facets_three_four_five}
	For every irreducible facet inducing Steiner graph $(G,T)$ with $|T|\in\{3,4,5\}$ the graph $G$ is a tree or a cactus (in fact, $G$ is an edge or a cycle). 
\end{proposition}

\begin{proof}
    \replaced[comment={\revrem{2}{37(a)}}]{Suppose that the claimed statement is not true. }{Supposing that the statement does not hold, we}
    \added{We then} choose $(G,T)$ to be 
	 an irreducible facet inducing Steiner graph  with 
	$
		\tau := |T|\in\{3,4,5\}
	$
	such that $G$ is neither a tree nor a cactus, where we make our choice such that $n := |V(G)|$ is 
 \replaced{as small as possible. }{minimal.} Let $c$ be facet weights for $(G,T)$ and $\gamma:=\gamma_c(G,T)$.

	Lemma~\ref{lem:bound}   provides the following upper bound on the number $m:=|E(G)|$ of edges of $G$:
	\begin{equation}
		\label{eq:m_ub}
		m \le n+\tau-3
\end{equation}

	On the other hand, as $G$ is two-node connected, but not a cycle, we have the  lower bound
	\begin{equation}
		\label{eq:m_lb1}
		m \ge n+1
	\end{equation}
	on the number of edges
	(see the remarks on ear decompositions below). Note that those two bounds already imply $\tau\ge 4$. 

	Clearly, due to the irreducibility, all terminal nodes have degree at least two and all non-terminal nodes have degree at least three. 
    In fact, by the minimality of $n$ we even deduce 
	\begin{equation}
		\label{eq:deg_bound_non_term}
		|\delta(v)| \ge 4\quad\text{for all }v \in V(G)\setminus T
	\end{equation}
    from Lemma~\ref{lem:YD_facet} and Lemma~\ref{lem:YD_no_tree_or_cactus}.  This in particular implies 
	\[
		2m = \sum_{v\in V(G)} |\delta(v)| \ge 2n + 2(n-\tau)\,,
	\]
	hence
	\begin{equation}
		\label{eq:m_lb2}
		m \ge 2n-\tau\,.
	\end{equation}
	Note that \eqref{eq:m_ub} and \eqref{eq:m_lb2} imply 
	    $n \le 2\tau-3$.

	The system \eqref{eq:m_ub}, \eqref{eq:m_lb1}, \eqref{eq:m_lb2} has the following seven integral solutions with  $3\le \tau\le 5$ and $n\ge \tau$:
\[
\begin{array}{c|ccc}
	& \tau & n & m \\
	\hline
	(1) & 4 & 4 & 5 \\
	(2) & 4 & 5 & 6 \\
	(3) & 5 & 5 & 6 \\
	(4) & 5 & 5 & 7 \\
	(5) & 5 & 6 & 7 \\
	(6) & 5 & 6 & 8 \\
	(7) & 5 & 7 & 9  
\end{array}
\]
The cases (1), (3), and (4) (i.e., the ones with $n=\tau$) can be ruled out immediately by the result of Conforti, Fiorini, and Pashkovich cited above. Indeed, in those cases we have $T = V(G)$, and due to $m\ge n+1$ (see~\eqref{eq:m_lb1}) there must be a node (thus, a terminal)  of degree at least three, which then implies $\gamma \ge 3$. Hence, according to Theorem~\ref{thm:CFP}, $G$ would have a prism or a pyramid as a minor in each of those cases, which, however, is not true as they all satisfy $n \le 5$.

In order to enumerate the possible graphs for each of the remaining four parameter combinations, we exploit the 
well-known fact that the two-node connected graph $G$ can be constructed by means of an \emph{ear decomposition} (Whitney~\cite{Whitney1931}): Starting from an arbitrary \emph{initial cycle} in $G$, we repeatedly add \emph{ears}, i.e.,  paths with at least one edge whose end nodes are disjoint nodes in the part of $G$ that has already been constructed, and whose inner nodes have not yet appeared in the construction so far. \added[comment={\revrem{2}{37d}}]{Note that the initial cycle is not considered to be an ear.} It is possible to  arrange the construction such that ears without inner nodes appear only after all nodes have shown up. The number of ears in an ear decomposition equals $m-n$. 
 
A first consequence of the existence of an ear-decomposition of $G$ is the following. As we have already ruled out the case $n = \tau$, the graph $G$ has at least one non-terminal node. Thus,  according to~\eqref{eq:deg_bound_non_term} it has a node with degree at least four, hence any ear decomposition  has at least two ears. From this we conclude
$m-n \ge 2$, 
which rules out  cases (2) and (5).

Therefore, we are left with the task to show that cases (6) and (7) cannot occur. As in both of those cases we have $m-n=2$, every ear decomposition of $G$ has exactly two ears. In  particular, no node  degree  exceeds four, thus all non-terminal nodes have degree equal to four (again, due to~\eqref{eq:deg_bound_non_term}). 
Furthermore,  both remaining cases  satisfy the equation
\begin{equation}
	\label{eq:equalityInCases}
	m = n+\tau-3\,,
\end{equation}
hence we know
\begin{equation}
	\label{eq:tightnessInCases}
		c(\delta(t)) = \gamma\quad\text{for all }t \in T
\end{equation}
according to Lemma~\ref{lem:bound}, i.e., every terminal singleton defines a root.
We conclude the proof by enumerating the graphs that remain to be considered and derive a specific contradiction for each of them.

\paragraph{Case (6):} $\tau = 5$, $n=6$, $m=8$

	Among the cycles that contain the non-terminal node $v \in V(G)\setminus T$, we choose one, say $C$,  of maximal length as the initial cycle of an ear decomposition of $G$. 
\replaced{If the cycle $C$ is  Hamiltonian, then $G$ consists of $C$ and two more edges $e_1$ and $e_2$ that share the non-terminal node. For  every root $\delta(S)$ both $S$ and its complement induce connected subgraphs of $G$ (see Proposition~\ref{prop:facets_first_properties}). This  implies $|\delta(S) \cap C| = 2$ for every root. Hence the facet defined by $(G,T)$ and $c$ is contained in the facet defined by $x(C)\ge 2$, which contradicts the fact $c(e_1),c(e_2) > 0$. }{Again, $C$ is not a Hamiltonian cycle (since the two edges in $E(G)\setminus C$ would both be incident to the non-terminal node of degree four, thus they would not cross w.r.t. $C$, contradicting Lemma~\ref{lem:Hamiltonian_cycle}).} 

Therefore,  we have $|C|\le 5$. From the maximality property of $C$ we find $|C|\ge 4$ \added[comment={\revrem{2}{37e}}]{(as adding an ear to a triangle creates a cycle of length at least four)}. 
In case of $|C|=4$, again due to the maximality property of $C$, $G$ must be the graph $G_1$, and if we have $|C|=5$, then $G$ must be one of the graphs $G_2$ and $G_3$ (see Fig.~\ref{fig:C}).

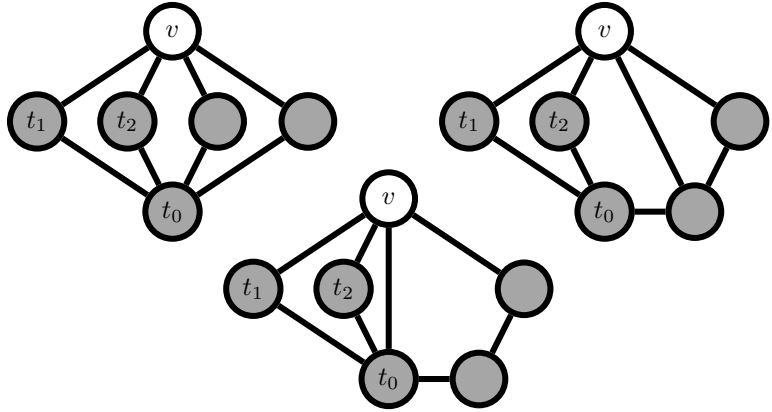
\begin{figure}[ht]
%	\centerline{%
%		\includegraphics[height=4cm]{Case6}
%	}
    \begingroup
    \renewcommand{\scale}{.6}
	\centerline{%
		\begin{tikzpicture}[scale=\scale]

%    \draw[step=.2cm,gray,very thin] (-3,-2) grid (3,2);
        
    \node [case_nonterm] (v) at (0,2) {$v$};
    \node [case_term] (t1) at (-3,0) {$t_1$};
    \node [case_term] (t2) at (-1,0) {$t_2$};
    \node [case_term] (t3) at (1,0) {};
    \node [case_term] (t4) at (3,0) {};
    \node [case_term] (t0) at (0,-2) {$t_0$};
    
    \draw [case_edge] (v) -- (t1);
    \draw [case_edge] (v) -- (t2);
    \draw [case_edge] (v) -- (t3);
    \draw [case_edge] (v) -- (t4);
    \draw [case_edge] (t0) -- (t1);
	\draw [case_edge] (t0) -- (t2);
	\draw [case_edge] (t0) -- (t3);
	\draw [case_edge] (t0) -- (t4);
\end{tikzpicture}
		\hspace{1cm}
		\begin{tikzpicture}[scale=\scale]

%    \draw[step=.2cm,gray,very thin] (-3,-2) grid (3,2);
        
    \node [case_nonterm] (v) at (0,2) {$v$};
    \node [case_term] (t1) at (-3,0) {$t_1$};
    \node [case_term] (t2) at (-1,0) {$t_2$};
    \node [case_term] (t3) at (2,-2) {};
    \node [case_term] (t4) at (3,0) {};
    \node [case_term] (t0) at (0,-2) {$t_0$};
    
    \draw [case_edge] (v) -- (t1);
    \draw [case_edge] (v) -- (t2);
    \draw [case_edge] (v) -- (t3);
    \draw [case_edge] (v) -- (t4);
    \draw [case_edge] (t0) -- (t1);
	\draw [case_edge] (t0) -- (t2);
	\draw [case_edge] (t0) -- (t3);
	\draw [case_edge] (t3) -- (t4);
\end{tikzpicture}
	}
	\vspace{-1cm}
	\centerline{%
		\begin{tikzpicture}[scale=\scale]

%    \draw[step=.2cm,gray,very thin] (-3,-2) grid (3,2);
        
    \node [case_nonterm] (v) at (0,2) {$v$};
    \node [case_term] (t1) at (-3,0) {$t_1$};
    \node [case_term] (t2) at (-1,0) {$t_2$};
    \node [case_term] (t3) at (2,-2) {};
    \node [case_term] (t4) at (3,0) {};
    \node [case_term] (t0) at (0,-2) {$t_0$};
    
    \draw [case_edge] (v) -- (t1);
    \draw [case_edge] (v) -- (t2);
    \draw [case_edge] (v) -- (t0);
    \draw [case_edge] (v) -- (t4);
    \draw [case_edge] (t0) -- (t1);
	\draw [case_edge] (t0) -- (t2);
	\draw [case_edge] (t0) -- (t3);
	\draw [case_edge] (t3) -- (t4);
\end{tikzpicture}
	}
	\endgroup
	\caption{The graphs $G_1$ (top left), $G_2$ (top right\added{)}, and $G_3$ (bottom). Terminals are gray.}
	\label{fig:C}
\end{figure}

Using~\eqref{eq:tightnessInCases}, we find that for each of the three graphs  
\[
	c(\delta(v)) \ge c(\delta(t_1)) + c(\delta(t_2)) - c(\delta(t_0)) = \gamma
\]
holds. 
Thus, in each of the three cases the inequality $cx\ge\gamma$ is 
\replaced[comment={\revrem{1}{15}}]{also }{even} valid for $\cutd(G)$ (as $\delta(v)$ is the only non-trivial cut in $G$ that is not a $T$-Steiner cut). Since none of the graphs $G_1$, $G_2$, and $G_3$ has a prism or a pyramid as a minor (due to $m=8$), Corollary~\ref{cor:CFP} implies \replaced[comment={\revrem{2}{37f}}]{$\gamma\le 2$}{ $\gamma = 2$}. However, each of the three graphs has a terminal $t$ of degree larger than two, which then  contradicts~\eqref{eq:tightnessInCases}. 

\paragraph{Case (7):} $\tau=5$, $n=7$, $m=9$

Due to $m=n+2$, and as 
	 both non-terminal nodes $v_1,v_2 \in V(G)\setminus T$ have degree equal to four, all terminal nodes have degrees equal to  two.
	 Therefore,  an arbitrary ear decomposition will have $v_1$ and $v_2$ in its initial cycle with both ears having $v_1$ and $v_2$ as their end nodes. 
	Thus, $G$ consists of four disjoint paths each having $v_1$ and $v_2$ as its end node, hence $G$ is one of the  graphs shown in Fig.~\ref{fig:D}.

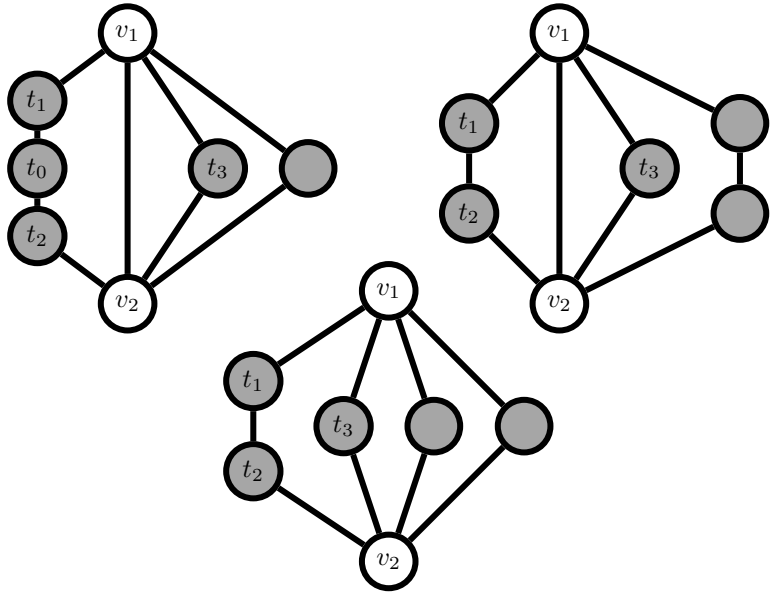
\begin{figure}[ht]
%	\centerline{%
%		\includegraphics[height=4cm]{Case7}	
%	}
    \begingroup
    \renewcommand{\scale}{.6}
	\centerline{%
	    \begin{tikzpicture}[scale=\scale]

%    \draw[step=.2cm,gray,very thin] (-3,-3) grid (3,3);
        
    \node [case_nonterm] (v1) at (0,3) {$v_1$};
    \node [case_nonterm] (v2) at (0,-3) {$v_2$};
    \node [case_term] (t1) at (-2,1.5) {$t_1$};
    \node [case_term] (t0) at (-2,0) {$t_0$};
    \node [case_term] (t2) at (-2,-1.5) {$t_2$};
    \node [case_term] (t3) at (2,0) {$t_3$};
    \node [case_term] (t4) at (4,0) {};
    
    \draw [case_edge] (v1) -- (v2);
    \draw [case_edge] (v1) -- (t1);
    \draw [case_edge] (v1) -- (t3);
    \draw [case_edge] (v1) -- (t4);
    \draw [case_edge] (v2) -- (t2);
    \draw [case_edge] (v2) -- (t3);
	\draw [case_edge] (v2) -- (t4);
	\draw [case_edge] (t0) -- (t1);
	\draw [case_edge] (t0) -- (t2);
\end{tikzpicture}
	    \hspace{1cm}
	    \begin{tikzpicture}[scale=\scale]

%    \draw[step=.2cm,gray,very thin] (-3,-3) grid (3,3);
        
    \node [case_nonterm] (v1) at (0,3) {$v_1$};
    \node [case_nonterm] (v2) at (0,-3) {$v_2$};
    \node [case_term] (t1) at (-2,1) {$t_1$};
    \node [case_term] (t2) at (-2,-1) {$t_2$};
    \node [case_term] (t4) at (4,1) {};
    \node [case_term] (t5) at (4,-1) {};
    \node [case_term] (t3) at (2,0) {$t_3$};
    
    \draw [case_edge] (v1) -- (v2);
    \draw [case_edge] (v1) -- (t1);
    \draw [case_edge] (v1) -- (t4);
    \draw [case_edge] (v1) -- (t3);
    \draw [case_edge] (v2) -- (t2);
    \draw [case_edge] (v2) -- (t5);
	\draw [case_edge] (v2) -- (t3);
	\draw [case_edge] (t1) -- (t2);
	\draw [case_edge] (t4) -- (t5);
\end{tikzpicture}
	}
    \vspace{-1cm}
    \centerline{%
    	\begin{tikzpicture}[scale=\scale]

%    \draw[step=.2cm,gray,very thin] (-3,-3) grid (3,3);
        
    \node [case_nonterm] (v1) at (0,3) {$v_1$};
    \node [case_nonterm] (v2) at (0,-3) {$v_2$};
    \node [case_term] (t1) at (-3,1) {$t_1$};
    \node [case_term] (t2) at (-3,-1) {$t_2$};
    \node [case_term] (t3) at (-1,0) {$t_3$};
    \node [case_term] (t4) at (1,0) {};
    \node [case_term] (t5) at (3,0) {};
    
    \draw [case_edge] (v1) -- (t1);
    \draw [case_edge] (v1) -- (t3);
    \draw [case_edge] (v1) -- (t4);
    \draw [case_edge] (v1) -- (t5);
    \draw [case_edge] (v2) -- (t2);
    \draw [case_edge] (v2) -- (t3);
	\draw [case_edge] (v2) -- (t4);
	\draw [case_edge] (v2) -- (t5);
	\draw [case_edge] (t1) -- (t2);

\end{tikzpicture}
    }
    \endgroup
\caption{The graphs $G_4$ (top left) , $G_5$ (top right), and $G_6$ (bottom). Terminals are dark.}
	\label{fig:D}
\end{figure}

From the equations in~\eqref{eq:tightnessInCases} we deduce
$
	c(t_1v_1) = c(t_2v_2)
$
for each of the graphs $G_4$, $G_5$, and $G_6$ (where for $G_4$ we additionally exploit $\added{c(t_1v_1)+c(t_0t_2)=}c(\delta\{t_0,t_1\})\ge \gamma$ and $\added{c(t_2v_2)+c(t_1t_0)=}c(\delta\{t_0,t_2\})\ge \gamma$\added[comment={\revrem{2}{37g}}]{, which in fact together with~\eqref{eq:tightnessInCases} implies $c(t_1v_1) = c(t_2v_2) = \gamma/2$}). 
For $G_4$ we thus have
$
	c(\delta(v_1)) = c(\delta(\{v_1,t_1,t_0,t_2\})) \ge \gamma
$
\replaced{and}{as well as}
$
c(\delta(v_2)) = c(\delta(\{v_2,t_2,t_0,t_1\})) \ge \gamma
$\replaced{. For both }{, and for}
$G_5$ and $G_6$ we find
$
c(\delta(v_1)) = c(\delta(\{v_1,t_1,t_2\})) \ge \gamma
$
as well as
$
c(\delta(v_2)) = c(\delta(\{v_2,t_2,t_1\})) \ge \gamma
$.
 
As additionally
$
	c(\delta\{v_1,v_2\}) \ge c(\delta(t_3)) = \gamma
$
holds for $G_4$, $G_5$, and $G_6$, the inequality $cx\ge\gamma$ is  in fact valid for $\cutd(G,V(G))$. Since none of the graphs $G_4$, $G_5$, and $G_6$  has a prism or a pyramid as a minor (as we have $m=9$, and each of the three graphs has only two nodes of degree larger than two), Corollary~\ref{cor:CFP}  implies \replaced[comment={\revrem{2}{37h}}]{$\gamma\le 2$}{ $\gamma = 2$} and $c$ is the all-one vector (since each edge is contained in at least one root\added[comment={\revrem{2}{37h}}]{, which has at least two edges since the graph is two-connected}). 

As for both $G_4$ and $G_5$ the edge $v_1v_2$ is not contained in any cut with two edges, we conclude $G=G_6$. However, the only cuts with exactly two edges in $G_6$ are the five cuts $\delta(t)$ for $t \in T$ and the cut $\delta(\{t_1,t_2\})$ which contradicts the existence of a root basis (due to $m=9$). 
\end{proof}

We conclude this section by 
the following characterization of  the facet defining inequalities of Steiner degree at most five for cut dominants that follows immediately from 
Theorem~\ref{thm:facets_three_four_five}.

\begin{definition}
	If $H$ is a cactus and $C$ is a cycle in $H$, then we denote by $\deg(C)$ the number of nodes in $C$ that are connected to nodes outside of $C$ (i.e., they are cut nodes of $G$),  and we define the \emph{defect} of $C$  to be 
	\begin{equation*}
		\max\{0,3-\deg(C)\}\,.
	\end{equation*}
\end{definition}

\begin{corollary}
	\label{cor:Steiner_degree_at_most_five}
	For each connected graph $G$ the facet defining inequalities of Steiner degree at most five for  $\cutd(G)$  are the inequalities
	\begin{equation*}
		x(E(H)) \quad \ge \quad 1
	\end{equation*}
	for each spanning tree $H$ in $G$ with at most five leaves, and the inequalities
	\begin{equation*}
		\sum_{\stackrel{e \in E(H)}{e \text{ is in some cycle of $H$}}} x_e 
			+\sum_{\stackrel{e \in E(H)}{e \text{ is in no cycle of $H$}}} 2\cdot x_e 
			\quad \ge \quad 2
	\end{equation*}
	for each spanning cactus $H$ of $G$ for which the number of leaves plus the sum of the defects of the cycles is at most five (which in particular implies that $H$ has no more than three cycles). 
\end{corollary} 

One consequence of Corollary~\ref{cor:Steiner_degree_at_most_five} is that the only vertices of the subtour relaxation polytope that have   Steiner degree at most five are the incidence vectors of the Hamiltonian cycles (i.e., the vertices of the traveling salesman polytope itself), which in fact have Steiner degree three.

\section{Conclusion}
\label{sec:conclusion}

\paragraph*{\bf More than five terminals.}

For six terminals or more the irreducible facet inducing Steiner graphs are considerably more involved than they are for up to five terminals (where we only have  single edges and cycles of length at most five according to Theorem~\ref{thm:facets_three_four_five}). In fact, at this point we do not know the  complete list of  irreducible facet inducing Steiner graphs with six terminals. Figure~\ref{fig:irr6} shows a complete list of those ones with six or seven nodes\added[comment={\revrem{2}{38}}]{ (obtained from enumeration via computer)}. 
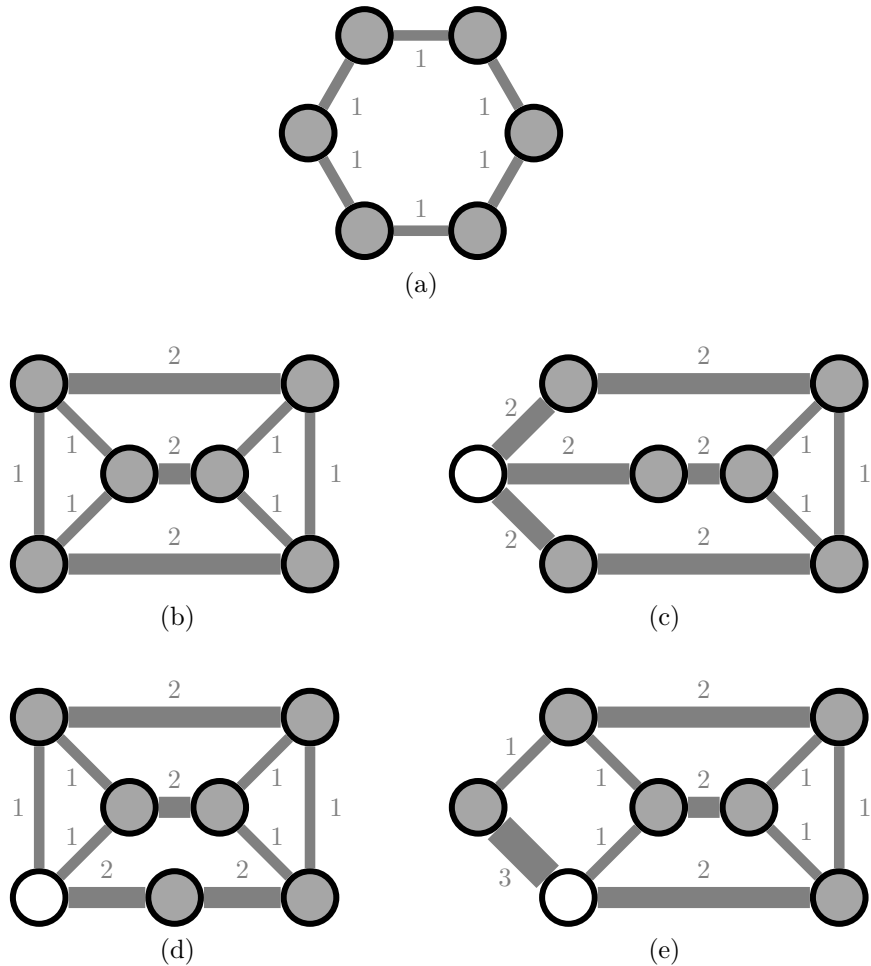
\begin{figure}[ht]
    \renewcommand{\scale}{.6}
    \centering
    \parbox[b]{.49\textwidth}{%
        \centering
        \begin{tikzpicture}[scale=\scale]
%,every edge quotes/.style={fill=white,font=\footnotesize}]]

%    \draw[step=.2cm,gray,very thin] (-2,-2) grid (2,2);
        
    \node [term] (v1) at (0:2.5) {};
    \node [term] (v2) at (60:2.5) {};
    \node [term] (v3) at (120:2.5) {};
    \node [term] (v4) at (180:2.5) {};
    \node [term] (v5) at (240:2.5) {};
    \node [term] (v6) at (300:2.5) {};

    \draw [edge1] (v1) edge["$1$"] (v2);
    \draw [edge1] (v2) edge["$1$"] (v3);
    \draw [edge1] (v3) edge["$1$"] (v4);
    \draw [edge1] (v4) edge["$1$"] (v5);
    \draw [edge1] (v5) edge["$1$"] (v6);
    \draw [edge1] (v6) edge["$1$"] (v1);
\end{tikzpicture}\\
        \centering (a)
    }\\
    \vspace{.5cm}%
    \parbox[b]{.49\textwidth}{%
        \centering
        \begin{tikzpicture}[scale=\scale]
%,every edge quotes/.style={fill=white,font=\footnotesize}]]

%    \draw[step=.2cm,gray,very thin] (-2,-2) grid (2,2);
        
    \node [term] (v1) at (0,2) {};
    \node [term] (v2) at (6,2) {};
    \node [term] (v3) at (2,0) {};
    \node [term] (v4) at (4,0) {};
    \node [term] (v5) at (0,-2) {};
    \node [term] (v6) at (6,-2) {};

    \draw [edge2] (v1) edge["$2$"] (v2);
    \draw [edge2] (v3) edge["$2$"] (v4);
    \draw [edge2] (v5) edge["$2$"] (v6);
    \draw [edge1] (v1) edge node [below=2pt, near start] {$1$} (v3);
    \draw [edge1] (v3) edge node [above=2pt, near end] {$1$} (v5);
    \draw [edge1] (v5) edge["$1$"] (v1);
    \draw [edge1] (v2) edge node [below=2pt, near start] {$1$} (v4);
    \draw [edge1] (v4) edge node [above=2pt, near end] {$1$} (v6);
    \draw [edge1] (v6) edge node [right=2pt] {$1$} (v2);
\end{tikzpicture}\\
        \centering (b)
    }\hfill%
    \parbox[b]{.49\textwidth}{%
        \centering
        \begin{tikzpicture}[scale=\scale]
%,every edge quotes/.style={fill=white,font=\footnotesize}]]

%    \draw[step=.2cm,gray,very thin] (-2,-2) grid (2,2);
        
    \node [term] (v1) at (0,2) {};
    \node [term] (v2) at (6,2) {};
    \node [term] (v3) at (2,0) {};
    \node [term] (v4) at (4,0) {};
    \node [term] (v5) at (0,-2) {};
    \node [term] (v6) at (6,-2) {};
    \node [nonterm] (v7) at (-2,0) {};

    \draw [edge2] (v1) edge["$2$"] (v2);
    \draw [edge2] (v3) edge["$2$"] (v4);
    \draw [edge2] (v5) edge["$2$"] (v6);
    \draw [edge2] (v7) edge node [above=2pt, near start] {$2$} (v1);
    \draw [edge2] (v7) edge node [above=0pt] {$2$} (v3);
    \draw [edge2] (v7) edge node [below=2pt, near start] {$2$} (v5);
    \draw [edge1] (v2) edge node [below=2pt, near start] {$1$} (v4);
    \draw [edge1] (v4) edge node [above=2pt, near end] {$1$} (v6);
    \draw [edge1] (v6) edge node [right=2pt] {$1$} (v2);
\end{tikzpicture}\\
  %       \centering\includegraphics[width=.3\textwidth]{SteinerCutDominants/figures/BASFAC_7_6_facet_1.pdf}\\
        \centering (c)
    }\\
    \vspace{.5cm}%
    \parbox[b]{.49\textwidth}{%
        \centering
        \begin{tikzpicture}[scale=\scale]
%,every edge quotes/.style={fill=white,font=\footnotesize}]]

%    \draw[step=.2cm,gray,very thin] (-2,-2) grid (2,2);
        
    \node [term] (v1) at (0,2) {};
    \node [term] (v2) at (6,2) {};
    \node [term] (v3) at (2,0) {};
    \node [term] (v4) at (4,0) {};
    \node [nonterm] (v5) at (0,-2) {};
    \node [term] (v6) at (6,-2) {};
    \node [term] (v7) at (3,-2) {};

    \draw [edge2] (v1) edge["$2$"] (v2);
    \draw [edge2] (v3) edge["$2$"] (v4);
    \draw [edge2] (v5) edge["$2$"] (v7);
    \draw [edge2] (v7) edge["$2$"] (v6);
    \draw [edge1] (v1) edge node [below=2pt, near start] {$1$} (v3);
    \draw [edge1] (v3) edge node [above=2pt, near end] {$1$} (v5);
    \draw [edge1] (v5) edge["$1$"] (v1);
    \draw [edge1] (v2) edge node [below=2pt, near start] {$1$} (v4);
    \draw [edge1] (v4) edge node [above=2pt, near end] {$1$} (v6);
    \draw [edge1] (v6) edge node [right=2pt] {$1$} (v2);
\end{tikzpicture}\\
   %     \centering\includegraphics[width=.3\textwidth]{SteinerCutDominants/figures/BASFAC_7_6_facet_2.pdf}\\
        \centering (d)
    }\hfill%
    \parbox[b]{.49\textwidth}{%
        \centering
        \begin{tikzpicture}[scale=\scale]
%,every edge quotes/.style={fill=white,font=\footnotesize}]]

%    \draw[step=.2cm,gray,very thin] (-2,-2) grid (2,2);
        
    \node [term] (v1) at (0,2) {};
    \node [term] (v2) at (6,2) {};
    \node [term] (v3) at (2,0) {};
    \node [term] (v4) at (4,0) {};
    \node [nonterm] (v5) at (0,-2) {};
    \node [term] (v6) at (6,-2) {};
    \node [term] (v7) at (-2,0) {};

    \draw [edge2] (v1) edge["$2$"] (v2);
    \draw [edge2] (v3) edge["$2$"] (v4);
    \draw [edge2] (v5) edge["$2$"] (v6);
    \draw [edge1] (v1) edge node [below=2pt, near start] {$1$} (v3);
    \draw [edge1] (v3) edge node [above=2pt, near end] {$1$} (v5);
    \draw [edge3] (v5) edge node [below=4pt, very near end] {$3$} (v7);
    \draw [edge1] (v7) edge node [above=2pt, near start] {$1$} (v1);
    \draw [edge1] (v2) edge node [below=2pt, near start] {$1$} (v4);
    \draw [edge1] (v4) edge node [above=4pt, near end] {$1$} (v6);
    \draw [edge1] (v6) edge node [right=2pt] {$1$} (v2);
\end{tikzpicture}\\
    %    \centering\includegraphics[width=.3\textwidth]{SteinerCutDominants/figures/BASFAC_7_6_facet_3.pdf}\\
        \centering (e)
    }
    \caption{The irreducible facet inducing Steiner graphs with six terminals and at most seven nodes (again, terminals are dark); the cycle (a) has right-hand-side two, the other ones (b)--(e) have right-hand-side four (note that (b) is a prism).}
    \label{fig:irr6}
\end{figure}
As one sees from the examples (c), (d), and (e) in Figure~\ref{fig:irr6}, in contrast to the situation with at most five terminals,  irreducible facet inducing Steiner graphs in general do have  non-terminal nodes as well. Note that (b) and (d) induce  vertices of the subtour relaxation polytope that have Steiner degree six.

\paragraph*{\bf (Non-)Uniqueness of facet weights.}
As a consequence of Theorem~\ref{thm:facets_three_four_five} (and the remarks made in connection with the definitions of Steiner trees and Steiner cacti inequalities) for up to five terminals the facet weights of each facet inducing Steiner graph are uniquely determined. 
\replaced[comment={\revrem{2}{40}}]{For more than five terminals the facet weights in general are no longer unique (see Figure~\ref{fig:nonunique_weights} for an example).}{We do not expect a similar result to hold for general numbers of terminals.}

\begin{figure}
    \renewcommand{\scale}{.6}
    \centering
    \begin{tikzpicture}[scale=\scale]

%    \draw[step=.2cm,gray,very thin] (-2,-2) grid (2,2);
        
    \node [term] (v5) at (0,0) {};
    \node [nonterm] (v0) at (12,0) {};
    \node [term] (v4) at (10,2) {};
    \node [nonterm] (v6) at (2,4) {};
    \node [term] (v9) at (4,4) {};
    \node [nonterm] (v3) at (8,4) {};
    \node [term] (v2) at (10,6) {};
    \node [term] (v7) at (0,8) {};
    \node [term] (v8) at (6,8) {};
    \node [nonterm] (v1) at (12,8) {};

    \draw [edge2] (v5) edge["$2$"] (v0);
    \draw [edge2] (v5) edge["$2$"] (v6);
    \draw [edge3] (v0) edge node [above=3pt, near start] {$3$} (v4);
    \draw [edge1] (v4) edge node [above=2pt, near start] {$1$} (v3);
    \draw [edge2] (v6) edge["$2$"] (v9);
    \draw [edge2] (v9) edge["$2$"] (v3);
    \draw [edge1] (v0) edge["$1$"] (v1);
    \draw [edge2] (v6) edge["$2$"] (v7);
    \draw [edge3] (v3) edge node [below=3pt, near end] {$3$} (v2);
    \draw [edge1] (v2) edge node [below=2pt, near end] {$1$} (v1);
    \draw [edge2] (v7) edge["$2$"] (v8);
    \draw [edge2] (v8) edge["$2$"] (v1);
\end{tikzpicture}
    \hfill
    \begin{tikzpicture}[scale=\scale]

%    \draw[step=.2cm,gray,very thin] (-2,-2) grid (2,2);
        
    \node [term] (v5) at (0,0) {};
    \node [nonterm] (v0) at (12,0) {};
    \node [term] (v4) at (10,2) {};
    \node [nonterm] (v6) at (2,4) {};
    \node [term] (v9) at (4,4) {};
    \node [nonterm] (v3) at (8,4) {};
    \node [term] (v2) at (10,6) {};
    \node [term] (v7) at (0,8) {};
    \node [term] (v8) at (6,8) {};
    \node [nonterm] (v1) at (12,8) {};

    \draw [edge3] (v5) edge["$3$"] (v0);
    \draw [edge1] (v5) edge["$1$"] (v6);
    \draw [edge2] (v0) edge node [above=2pt, near start] {$2$} (v4);
    \draw [edge2] (v4) edge node [above=2pt, near start] {$2$} (v3);
    \draw [edge2] (v6) edge["$2$"] (v9);
    \draw [edge2] (v9) edge["$2$"] (v3);
    \draw [edge1] (v0) edge["$1$"] (v1);
    \draw [edge3] (v6) edge["$3$"] (v7);
    \draw [edge2] (v3) edge node [below=2pt, near end] {$2$} (v2);
    \draw [edge2] (v2) edge node [below=2pt, near end] {$2$} (v1);
    \draw [edge1] (v7) edge["$1$"] (v8);
    \draw [edge3] (v8) edge["$3$"] (v1);
\end{tikzpicture}  
    \caption{An example of a pair of different facet weights for the same irreducible facet inducing Steiner graph on 10 nodes and six terminals. The example has been determined via computer search.} 
    \label{fig:nonunique_weights}
\end{figure}

%\todo[inline]{The following refuses Michele's conjecture formulated in Sect.~\ref{sec:basics} before. THAT WAS NOT THE CONJECTURE!!!}
%If $(G,T)$ is a facet graph with facet weights $c$, then the minimum weight of a cut in $G$ that is not a Steiner cut can be smaller than $\gamma_c(G,T)$ (although our resuts will show that this does not happen for $|T|\le 5$). One example is depicted in the right part of the following figure  (with $T=\{1,\dots,6\}$), which shows a facet defining inequality as it arises from the facet defining inequality indicated in the  left part  by subdividing an edge.
%\begin{center}
%    \includegraphics[width=\textwidth]{SteinerCutDominants/figures/ConjectureMichele.pdf}
%\end{center}

\paragraph*{\bf Upwards validity.} 

Turning a non-terminal node into a terminal can turn a 
\replaced{facet defining }{facet-defining}
Steiner cut inequality into an invalid one, as the inequality arising from (b) in Figure~\ref{fig:irr6} by subdividing one of the triangle-edges by a non-terminal node shows. However, we are only aware of \replaced[comment={\revrem{2}{42}}]{examples }{exmples} exhibiting that effect where the non-terminal node has degree two. Therefore, one might ask the question whether every \emph{irreducible} facet defining Steiner cut inequality for $\cutd{(G,T)}$ in fact is valid (and thus facet defining) for $\cutd{(G)}$. Again, Theorem~\ref{thm:facets_three_four_five} at least shows that this holds for $|T|\le 5$. 

\paragraph*{\bf Computing the facets in polynomial time.}
It appears conceivable that one can, for every given Steiner graph $(G,T)$ with $|T|\le 5$, compute all Steiner subtrees and Steiner subcactii of $(G,T)$ in time that is bounded polynomially in the total size of in- and output (for a survey on enumerating $s$-$t$-paths and other structures see~ \cite{Grossi2016}). This would imply via  
Theorem~\ref{thm:facets_three_four_five}  that for Steiner graphs $(G,T)$ with $|T|\le 5$ one can compute a list \replaced[comment={\revrem{2}{43}}]{of }{if} all facet defining inequalities for $\cutd{(G,T)}$ in output-polynomial time. A general question would be whether Theorem~\ref{thm:bounded_number_terminals_irreducible} opens up possibilites for a corresponding output-polynomial time algorithm for every fixed number of terminals.  

\paragraph*{\bf More powerful operations.}
Theorem~\ref{thm:facets_three_four_five} shows that for each Steiner graph $(G,T)$ with $|T|\le 5$ the facet defining inequalities for $\cutd{(G,T)}$ can be constructed from the facet defining inequalities for $\cutd(K_{\tau})$ with $\tau \le |T|$
via iterated applications of the two operations \emph{gluing} and \emph{subdividing}. For more than five terminals, the corresponding result is not true, in general. Thus, the question for a more powerful set of operations arises that would allow for a similar result for arbitrary numbers of terminals. In fact, it may appear tempting to believe that the polyhedral structure of $\cutd{(G,T)}$ in some sense is a refinement of  the polyhedral structure of $\cutd{(K_{|T|})}$.

\ifMOR
    \bibliographystyle{informs2014} 
    \bibliography{SteinerCutDominants} 
\else
    \printbibliography
\fi

\end{document}